\documentclass[12pt,letterpaper,british,american]{article}
\usepackage[utf8]{inputenc}
\usepackage{color}
\usepackage{babel}
\usepackage{amsmath}
\usepackage{amsthm}
\usepackage{amssymb}
\usepackage{graphicx}
\PassOptionsToPackage{normalem}{ulem}
\usepackage{ulem}
\usepackage[unicode=true,pdfusetitle,
 bookmarks=true,bookmarksnumbered=false,bookmarksopen=false,
 breaklinks=false,pdfborder={0 0 1},backref=page,colorlinks=true]
 {hyperref}
\hypersetup{
 linkcolor=blue, citecolor=blue}

\makeatletter

\pdfpageheight\paperheight
\pdfpagewidth\paperwidth

\numberwithin{equation}{section}
\numberwithin{figure}{section}
\theoremstyle{plain}
\newtheorem{thm}{\protect\theoremname}[section]
\theoremstyle{plain}
\newtheorem{cor}[thm]{\protect\corollaryname}
\theoremstyle{plain}
\newtheorem{prop}[thm]{\protect\propositionname}
\theoremstyle{plain}
\newtheorem{lem}[thm]{\protect\lemmaname}
\theoremstyle{remark}
\newtheorem{rem}[thm]{\protect\remarkname}
\theoremstyle{definition}
\newtheorem{defn}[thm]{\protect\definitionname}

\@ifundefined{date}{}{\date{}}
\usepackage{graphicx}

\usepackage{enumitem}
\usepackage[backref=page]{hyperref}

\DeclareMathOperator\tr{tr}
\DeclareMathOperator\Tr{Tr}

\makeatother

\addto\captionsamerican{\renewcommand{\corollaryname}{Corollary}}
\addto\captionsamerican{\renewcommand{\definitionname}{Definition}}
\addto\captionsamerican{\renewcommand{\lemmaname}{Lemma}}
\addto\captionsamerican{\renewcommand{\propositionname}{Proposition}}
\addto\captionsamerican{\renewcommand{\remarkname}{Remark}}
\addto\captionsamerican{\renewcommand{\theoremname}{Theorem}}
\addto\captionsbritish{\renewcommand{\corollaryname}{Corollary}}
\addto\captionsbritish{\renewcommand{\definitionname}{Definition}}
\addto\captionsbritish{\renewcommand{\lemmaname}{Lemma}}
\addto\captionsbritish{\renewcommand{\propositionname}{Proposition}}
\addto\captionsbritish{\renewcommand{\remarkname}{Remark}}
\addto\captionsbritish{\renewcommand{\theoremname}{Theorem}}
\providecommand{\corollaryname}{Corollary}
\providecommand{\definitionname}{Definition}
\providecommand{\lemmaname}{Lemma}
\providecommand{\propositionname}{Proposition}
\providecommand{\remarkname}{Remark}
\providecommand{\theoremname}{Theorem}

\begin{document}
\global\long\def\F{\mathcal{F} }%
\global\long\def\Aut{\mathrm{Aut}\chi}%
\global\long\def\C{\mathbb{C}}%
\global\long\def\H{\mathcal{H}}%
\global\long\def\U{\mathrm{U}}%
\global\long\def\P{\mathcal{P}}%
\global\long\def\ext{\mathrm{ext}}%
\global\long\def\hull{\mathrm{hull}}%
\global\long\def\triv{\mathrm{triv}}%
\global\long\def\Hom{\mathrm{Hom}}%

\global\long\def\trace{\mathrm{tr}}%
\global\long\def\End{\mathrm{End}}%

\global\long\def\L{\mathcal{L}}%
\global\long\def\W{\mathcal{W}}%
\global\long\def\E{\mathbb{E}}%
\global\long\def\SL{\mathrm{SL}}%
\global\long\def\R{\mathbb{R}}%
\global\long\def\Z{\mathbb{Z}}%
\global\long\def\rs{\to}%
\global\long\def\A{\mathcal{A}}%
\global\long\def\a{\mathbf{a}}%
\global\long\def\rsa{\rightsquigarrow}%
\global\long\def\D{\mathbf{D}}%
\global\long\def\b{\mathbf{b}}%
\global\long\def\df{\mathrm{def}}%
\global\long\def\eqdf{\stackrel{\df}{=}}%
\global\long\def\ZZ{\mathcal{Z}}%
\global\long\def\Tr{\mathrm{Tr}}%
\global\long\def\N{\mathbb{N}}%
\global\long\def\std{\mathrm{std}}%
\global\long\def\HS{\mathrm{H.S.}}%
\global\long\def\e{\varepsilon}%
\global\long\def\c{\mathbf{c}}%
\global\long\def\d{\mathbf{d}}%
\global\long\def\AA{\mathbf{A}}%
\global\long\def\BB{\mathbf{B}}%
\global\long\def\u{\mathbf{u}}%
\global\long\def\v{\mathbf{v}}%
\global\long\def\spec{\mathrm{spec}}%
\global\long\def\Ind{\mathrm{Ind}}%
\global\long\def\half{\frac{1}{2}}%
\global\long\def\Re{\mathrm{Re}}%
\global\long\def\Im{\mathrm{Im}}%
\global\long\def\p{\mathfrak{p}}%
\global\long\def\j{\mathbf{j}}%
\global\long\def\uB{\underline{B}}%
\global\long\def\tr{\mathrm{tr}}%
\global\long\def\rank{\mathrm{rank}}%
\global\long\def\K{\mathcal{K}}%
\global\long\def\hh{\mathcal{H}}%
\global\long\def\h{\mathfrak{h}}%

\global\long\def\EE{\mathcal{E}}%
\global\long\def\PSL{\mathrm{PSL}}%
\global\long\def\G{\mathcal{G}}%
\global\long\def\Int{\mathrm{Int}}%
\global\long\def\acc{\mathrm{acc}}%
\global\long\def\awl{\mathsf{awl}}%
\global\long\def\even{\mathrm{even}}%
\global\long\def\z{\mathbf{z}}%
\global\long\def\id{\mathrm{id}}%
\global\long\def\CC{\mathcal{C}}%
\global\long\def\cusp{\mathrm{cusp}}%
\global\long\def\new{\mathrm{new}}%

\global\long\def\LL{\mathbb{L}}%
\global\long\def\M{\mathbf{M}}%
\global\long\def\I{\mathcal{I}}%
\global\long\def\X{X}%
\global\long\def\free{\mathbf{F}}%
\global\long\def\into{\hookrightarrow}%
\global\long\def\Ext{\mathrm{Ext}}%
\global\long\def\B{\mathcal{B}}%
\global\long\def\Id{\mathrm{Id}}%
\global\long\def\Q{\mathbb{Q}}%

\global\long\def\O{\mathcal{O}}%
\global\long\def\Mat{\mathrm{Mat}}%
\global\long\def\NN{\mathrm{NN}}%
\global\long\def\nn{\mathfrak{nn}}%
\global\long\def\Tr{\mathrm{Tr}}%
\global\long\def\SGRM{\mathsf{SGRM}}%
\global\long\def\m{\mathbf{m}}%
\global\long\def\n{\mathbf{n}}%
\global\long\def\k{\mathbf{k}}%
\global\long\def\GRM{\mathsf{GRM}}%
\global\long\def\vac{\mathrm{vac}}%
\global\long\def\SS{\mathcal{S}}%
\global\long\def\red{\mathrm{red}}%
\global\long\def\V{V}%
\global\long\def\SO{\mathrm{SO}}%
\global\long\def\Gd{\Gamma^{\vee}}%
\global\long\def\fd{\mathrm{fd}}%
\global\long\def\perm{\mathrm{perm}}%
\global\long\def\tos{\xrightarrow{\mathrm{strong}}}%
\global\long\def\HH{\mathbb{H}}%
\global\long\def\T{\mathcal{T}}%
\global\long\def\Jac{\mathsf{Jac}}%
\global\long\def\Ham{\mathsf{Ham}}%
\global\long\def\chit{\chi}%
\global\long\def\fix{\mathsf{fix}}%

\global\long\def\Cnk{\left(\C^{n}\right)^{\otimes k}}%
\global\long\def\scl{\mathrm{scl}}%
\global\long\def\F{\mathbf{F}}%
\global\long\def\prob{\mathbb{P}}%
\global\long\def\cU{\mathcal{U}}%
\global\long\def\SU{\mathrm{SU}}%

\vspace{-5in} 
\title{Strong asymptotic freeness of Haar unitaries in quasi-exponential
dimensional representations}
\author{Michael Magee and Mikael de la Salle}
\maketitle
\begin{abstract}
We prove almost sure strong asymptotic freeness of i.i.d. random unitaries
with the following law: sample a Haar unitary matrix of dimension
$n$ and then send this unitary into an irreducible representation
of $\U(n)$. The strong convergence holds as long as the irreducible
representation arises from a pair of partitions of total size at most
$n^{\frac{1}{42}-\varepsilon}$ and is uniform in this regime. 

Previously this was known for partitions of total size up to $\asymp\log n/\log\log n$
by a result of Bordenave and Collins. 
\end{abstract}
\tableofcontents{}

\section{Introduction}

Let $\U(n)$ denote the group of complex $n\times n$ unitary matrices.
For each $k,\ell\in\N$ there is a unitary representation
\[
\pi_{k,\ell}^{0}:\U(n)\to\U\left(\left(\C^{n}\right)^{\otimes k}\otimes\left(\left(\C^{n}\right)^{\vee}\right)^{\otimes\ell}\right).
\]
This representation has a non-zero invariant vector for $\U(n)$ if
and only if $k=\ell$, and in that case the space of invariant vectors
is understood\footnote{The space of invariant vectors is the image of $\C[S_{k}]$ in $\End\left((\C^{n})^{\otimes k}\right)\cong\left(\C^{n}\right)^{\otimes k}\otimes\left(\left(\C^{n}\right)^{\vee}\right)^{\otimes k}$.
If $n\geq k$ this is isomorphic to $\C[S_{k}]$ and we will always
be in this regime in this paper. If $n<k$ the dimension of the space
is still understood as the number of permutations in $S_{k}$ with
longest increasing subsequence $\leq n$ \cite[\S 8]{BaikRains}.}. Let $\pi_{k,\ell}$ be the restriction of $\pi_{k,\ell}^{0}$ to
the orthocomplement to the invariant vectors.

Fix $r\in[n]$. The main theorem of this paper is the following.
\begin{thm}
\label{thm:main}Let $U_{1}^{(n)},\ldots,U_{r}^{(n)}$ denote i.i.d.
Haar distributed elements of $\U(n)$. For any $A<\frac{1}{42}$ the
following holds almost surely. For any non-commutative $*$-polynomial\footnote{That is, a polynomial in $r$ non-commuting indeterminates $X_{1},\ldots,X_{r}$
and their formal adjoints $X_{1}^{*},\ldots,X_{r}^{*}$.} $p$
\[
\sup_{k+\ell\leq n^{A}}\left|\left\Vert \pi_{k,\ell}\left(p\left(U_{1}^{(n)},\ldots,U_{r}^{(n)}\right)\right)\right\Vert -\|p(x_{1},\ldots,x_{r})\|\right|=o(1)
\]
where $x_{1},\ldots,x_{r}$ are generators of a free group $\F_{r}$
and the norm on the right is the one in $C_{\red}^{*}\left(\F_{r}\right)$.
\end{thm}

An analog of Theorem \ref{thm:main} with $\U(n)$ replaced by $S_{n}$ 
has been obtained by Cassidy \cite{cassidy2023projectionformulasrefinementschurweyljones, cassidy2025randompermutationsactingktuples}, building in part on the current work.

Theorem \ref{thm:main} gives a clean statement about strong convergence
of i.i.d. Haar elements of $\SU(n)$ in representations of quasi-exponential
dimension in the following form:
\begin{cor}
\label{cor:mainSUn}Let $U_{1}^{(n)},\ldots,U_{r}^{(n)}$ denote i.i.d.
Haar distributed elements of $\SU(n)$. For any $A<\frac{1}{42}$
the following holds almost surely. For any non-commutative $*$-polynomial
$p$
\[
\sup_{\text{ }\substack{\pi\in\widehat{\SU(n)}\backslash\triv\\
\dim(\pi)\leq\exp\left(n^{A}\right)
}
}\left|\left\Vert \pi\left(p\left(U_{1}^{(n)},\ldots,U_{r}^{(n)}\right)\right)\right\Vert -\|p(x_{1},\ldots,x_{r})\|\right|=o(1).
\]
\end{cor}

Our proof of Theorem \ref{thm:main} relies on two things:
\begin{itemize}
\item a recent breakthrough of Chen, Garza-Vargas, Tropp, and van Handel
\cite{chen2024new} who found a new (and remarkable) approach to strong
convergence based on `differentiation with respect to $n^{-1}$'.
We add to this method in various ways in the sequel. One of the conceptual
differences is a new criterion that we find for temperedness of a
unitary representation of a free group (or more generally a group
with the rapid decay property), and therefore for strong convergence
towards the regular representation, see \S \ref{sec:Temperdness_criterion}.
This allows us to bypass the use of Pisier's linearization argument,
which was essential in many proofs of strong convergence so far (a
notable exception is the work of Paraud and Collins-Guionnet-Parraud
\cite{CollGuiParr22,parraud2023asymptotic,ParraudCMP,ParraudPTRF}),
and replace it by considerations on random walks on free groups, see
\S \ref{sec:RW}.
\item The method above is given as input rapid decay estimates for expected
values of stable characters of word maps on $\U(n)$. The key feature
of these estimates is that, up to a point, they actually improve with
the complexity of the representations that we consider. This feature
is intimately related to the concept of \emph{stable commutator length}
in free groups as uncovered in \cite{MageePuder1}.
\end{itemize}
Theorem \ref{thm:main} was proved when $k=1,\ell=0$ by Collins and
Male in \cite{Collins2014}. This resolved a problem left open in
Haagerup and Thorbjørnsen's breakthrough work \cite{HaagerupThr}.
Bordenave and Collins \cite{bordenave2022strong} prove Theorem \ref{thm:main}
in the regime 
\[
k+\ell\leq c\frac{\log(n)}{\log\log(n)}
\]
for some positive $c$. 

Theorem~\ref{thm:main} can be written in the following two alternate
ways.
\begin{itemize}
\item Let $V_{i}^{(n)}\eqdf\bigoplus_{k+\ell\leq n^{A}}\pi_{k,\ell}(U_{i}^{(n)})$.
The random matrices $V_{i}^{(n)}$ are (almost surely)\emph{ strongly
asymptotically free}.
\item The random unitary representations of $\F_{r}$ described by $x_{i}\mapsto V_{i}^{(n)}$
almost surely \emph{strongly converge} to the regular representation
of $\F_{r}$.
\end{itemize}
The following simple case of Theorem \ref{thm:main} is both new and
important.
\begin{cor}
\label{cor:Hecke-operator}For any $A<\frac{1}{42},$ almost surely,
for all $k,\ell\in\N\cup\{0\}$ such that $k+\ell\leq n^{A}$
\[
\left\Vert \sum_{i=1}^{r}\pi_{k,\ell}\left(U_{i}^{(n)}\right)+\pi_{k,\ell}\left(U_{i}^{(n)}\right)^{-1}\right\Vert =2\sqrt{2r-1}+o_{n\to\infty}(1).
\]
\end{cor}

We highlight this corollary in connection with the following well
known question of Gamburd, Jakobson, and Sarnak \cite[pg. 57]{GamburdJakobsonSarnak}
--- do generic in (Haar) measure $(u_{1},\ldots,u_{r})\in\mathrm{SU}(2)^{r}$
have some $\epsilon>0$ such that for all $k>0$
\[
\left\Vert \sum_{i=1}^{r}\pi_{k}\left(u_{i}\right)+\pi_{k}\left(u_{i}\right)^{-1}\right\Vert \leq2r-\epsilon?
\]

It was proved by Bourgain and Gamburd in \cite{BourgainGamburdSU2}
for $u_{i}$ with algebraic entries that generate a dense subgroup
of $\mathrm{SU}(2)$. This was extended to $\mathrm{SU}(n)$ for general
$n\geq3$ by Bourgain and Gamburd \cite{BourgainGamburdSUd} ---
in this case the spectral gap is uniform over all representations
$\pi_{k,\ell}$ with $k+\ell>0$. The analogous result for all compact
simple Lie groups was obtained by Benoist and de Saxce \cite{BenoistdeSaxce}.

Of course, Corollary \ref{cor:Hecke-operator} does not advance these
questions for any fixed $n$ as it deals only with infinite sequences
of random matrices of dimension $n\to\infty$ but it does make significant
progress on a relaxation of the problem. The motif of Theorem \ref{thm:main}
is that we have only a small amount of randomness of our matrices
compared to their dimensions.

The method of proof also allows to obtain estimates for matrix coefficients. For example, we prove
\begin{thm}\label{thm:large_coefficients_standard} Let $k_n=\exp(n^{\frac 1 2} (\log n)^{-4})$. For every $q$ and every sequence $P_n$ of non-commutative $*$-polynomial of degree $\leq q$ and with coefficients in $M_{k_n}(\C)$, almost surely
  \[ \lim_n \frac{\|P_n(U_1^{(n)},\dots,U_r^{(n)})\|}{\|P_n(x_1,\dots,x_r)\|}=1.\]
  \end{thm}
Bordenave and Collins \cite[Theorem 1.1]{BordenaveCollins3} proved the same theorem for $q=1$ and $k_n=\exp(n^{\frac{1}{32r+160}})$ instead of $k_n=\exp(n^{\frac 1 2 + o(1)})$; this result gives a third proof of the result from \cite{belinschi2024strongconvergencetensorproducts}, which has important consequences in the theory of von Neumann algebras \cite{zbMATH07565550,hayes2024consequencesrandommatrixsolution}. It is an intriguing question whether the exponent $\frac 1 2$ is optimal or not. As observed by Pisier \cite{pisiersubexponential}, the optimal exponent is necessarily at most $2$.
\subsection*{Acknowledgments}

We give huge thanks to Ramon van Handel who generously explained the
methods of the work \cite{chen2024new} to us during several enjoyable
meetings at the I.A.S. in April/May 2024. Without these meetings,
this paper would not exist.

Funding: 

M. M. This material is based upon work supported by the National Science
Foundation under Grant No. DMS-1926686. This project has received
funding from the European Research Council (ERC) under the European
Union\textquoteright s Horizon 2020 research and innovation programme
(grant agreement No 949143).

M. S. Research supported by the Charles Simonyi Endowment at the Institute
for Advanced Study, and the ANR project ANCG Project-ANR-19-CE40-0002.

\section{Preliminaries}

\subsection{Representation theory of $\protect\U(n)$\label{subsec:Representation-theory-of}}

For partitions $\lambda=(\lambda_{1},\ldots,\lambda_{p})\vdash k,\mu=(\mu_{1},\ldots,\mu_{q})\vdash\ell$,
for $n\geq p+q$ let $s_{\lambda,\mu}:\U(n)\to\C$ denote the character
of the representation 
\[
(\pi_{\lambda,\mu},V^{\lambda,\mu})
\]
with dominant weight
\begin{align*}
(\lambda_{1},\lambda_{2},\ldots,\lambda_{p},\underbrace{0,0,\ldots,0} & ,-\mu_{q},-\mu_{q-1},\ldots,-\mu_{1}).\\
n-(p+q)
\end{align*}
Because this character interpolates for all $n\geq p+q$, we refer
to it as a \emph{stable }character. In the paper we write $|\lambda|=k$
for the size of $\lambda$.

\subsection{Representation theory of $\protect\SU(n)$}

The irreducible unitary representations of $\SU(n)$ are closely related
to those of $\U(n)$: for every partitions $\lambda,\mu$ as before
with $n\ge p+q$, the restriction of $\pi_{\lambda,\mu}$ to $\SU(n)$
is an irreducible representation; all irreducible representations
of $\SU(n)$ appear in this way; and the restrictions of $\pi_{\lambda,\mu}$ to
$\SU(n)$ and $\pi_{\lambda',\mu'}$ to $\SU(n)$ coincide if and only
if the dominant weights differ by a multiple of $(1,\dots,1)$. In
other words, the irreducible unitary representations of $\SU(n)$
are indexed by the sets of integral dominant weights of the root system
$\mathrm{A}_{n-1}$, which can be parametrized by the set of $n$-tuples
of integers $\Lambda=(\Lambda_{1},\dots,\Lambda_{n})\in\Z^{n}$ with
$\Lambda_{1}\geq\dots\geq\Lambda_{n}$, modulo the subgroup $\Z(1,\dots,1)$. 

Write $\|\Lambda\|_{1}$ for the norm on $\R^{n}/\R(1,\dots,1)$

\[
\|\Lambda\|_{1}=\inf_{t\in\R}\sum_{i}|\Lambda_{i}-t|.
\]

It is the natural norm in the dual of the subspace of $(\R^{n},\|\cdot\|_{\infty})$
for which the coordinates sum to $0$. In what follows, whenever $\Lambda$
belongs to the quotient $\R^{n}/\big(\R(1,\dots,1)\big)$, we will
chose a representative $(\Lambda_{1},\dots,\Lambda_{n})\in\R^{n}$
such that $\sum_{i}|\Lambda_{i}|=\|\Lambda\|_{1}$. If $\Lambda$
belongs to the image of $\Z^{n}$, we can and will assume that $(\Lambda_{1},\dots,\Lambda_{n})$
has integer coordinates.

By Weyl's dimension formula, if $\pi$ is an irreducible representation
of $\SU(n)$ with highest weight $\Lambda$, 
\begin{equation}
\dim(\pi)=\prod_{1\leq i<j\leq n}\frac{j-i+\Lambda_{i}-\Lambda_{j}}{j-i}.\label{eq:Weyl_dimension_formula}
\end{equation}
Lower bounds for $\dim(\pi)$ in terms of $\Lambda$ have been obtained
in \cite{GuralnickLarsenManack}, but we will need a bound that is
more precise\footnote{For example if $\Lambda=(k,0,\dots,0)$, our bound gives $\dim(\pi)\geq\exp(ck)$,
whereas \cite{GuralnickLarsenManack} gives $\dim(\pi)\geq(k+1)^{\log(n-1)}$.} in the regime $\|\Lambda\|_{1}=o(n)$. 
\begin{prop}
There is a constant $c>0$ such that for every integer $n\geq2$ and
every irreducible representation $\pi$ of $\SU(n)$, 
\[
\exp(c\min(\|\Lambda\|_{1},n))\leq\dim(\pi)\leq\exp(\|\Lambda\|_{1}\log n),
\]
where $\Lambda$ is the highest weight of $\pi$. 
\end{prop}

\begin{proof}
The upper bound is because $\pi$ appears as a sub-representation
of the representation of $\SU(n)$ on $(\C^{n})^{k}\otimes(\overline{\C^{n}})^{\otimes\ell})$,
where $k$ is the sum of the positive $\Lambda_{i}$'s and $-\ell$
is the sum of the negative $\Lambda_{i}$'s. And $(\C^{n})^{k}\otimes(\overline{\C^{n}})^{\otimes\ell})$
has dimension $n^{k+\ell}=\exp(\|\Lambda\|_{1}\log n)$.

For the lower bound, by Weyl's dimension formula (\ref{eq:Weyl_dimension_formula})
we have to prove 
\begin{equation}
c\min(\|\Lambda\|_{1},n)\leq\sum_{i<j}\log\left(1+\frac{\Lambda_{i}-\Lambda_{j}}{j-i}\right).\label{eq:goal_upper_bound_dim}
\end{equation}
We will prove this inequality for every $\Lambda\in\R^{n}$ with $\Lambda_{1}\geq\dots\geq\Lambda_{n}$.
In that case, since the right-hand side increases if we replace $\Lambda$
by $t\Lambda$ for $t>1$, it is enough to consider the case when
$\|\Lambda\|_{1}\leq n$. Finally, replacing $\Lambda$ by $(-\Lambda_{n},\dots,-\Lambda_{1})$,
we can assume that the sum of the positive entries of $\Lambda$ is
at least $\|\Lambda\|_{1}/2$.

Let $n_{1}=\lfloor\frac{n}{2}\rfloor$ and $n_{2}=\lfloor\frac{3n}{4}\rfloor$.
Expressing that the $\ell_{1}$ norm of $(\Lambda_{1},\dots,\Lambda_{n})$
is at most the $\ell_{1}$ norm of $(\Lambda_{1}-t,\dots,\Lambda_{n}-t)$
for every $t\neq0$, we see that $\Lambda$ has at most $\frac{n}{2}$
positive entries, and at most $\frac{n}{2}$ negative entries. In
particular, we have 
\begin{equation}
\sum_{i=1}^{n_{1}}\Lambda_{i}\geq\frac{1}{2}\|\Lambda\|_{1},\label{eq:ell1_norm_of_Lambda_almost_sum_of_positives}
\end{equation}
and $\Lambda_{j}\leq0$ for every $n_{2}<j\leq n$. If $i\leq n_{1}$
and $j>n_{2}$, we have 
\[
\frac{\Lambda_{i}-\Lambda_{j}}{j-i}\geq\frac{\Lambda_{i}}{n}.
\]
We deduce the obvious bound 
\[
\sum_{i<j}\log\Big(1+\frac{\Lambda_{i}-\Lambda_{j}}{j-i}\Big)\geq\sum_{i\leq n_{1},n_{2}<j}\log\Big(1+\frac{\Lambda_{i}}{n}\Big).
\]
By the bound $\log(1+t)\geq t\log(2)$ for every $0\leq t\leq1$ and
the inequality $0\leq\Lambda_{i}\leq\|\Lambda\|_{1}=n$, this is at
most 
\[
\log(2)\sum_{1\leq i\leq n_{1}}\sum_{n_{2}<j\leq n}\frac{\Lambda_{i}}{n}\geq\frac{\log(2)}{4}\sum_{i=1}^{n_{1}}\Lambda_{i}.
\]
By (\ref{eq:ell1_norm_of_Lambda_almost_sum_of_positives}), we obtain
(\ref{eq:goal_upper_bound_dim}) with $c=\frac{\log(2)}{8}$.
\end{proof}
\begin{cor}
\label{cor:dimensions_SUn_representations}For every $0<A\leq1$ and
every $n\geq N$, for every irreducible representation $\pi$ of $\SU(n)$:
\[
\dim(\pi)<\exp(cn^{A})
\]
implies 
\[
\pi\subset\bigoplus_{k+\ell\leq n^{A}}\pi_{k,\ell},
\]
which in turn implies 
\[
\dim(\pi)\leq\exp(n^{A}\log n).
\]
\end{cor}

\section{Matrix integral results\label{sec:Matrix-integral-results}}

Let $\F_{r}$ denote the free group on a fixed generating set $X\eqdf\{x_{1},\ldots,x_{r}\}$.
For $w\in\F_{r}$ let 
\[
w:\U(n)^{r}\to\U(n)
\]
denote the induced word map defined by substituting in an $r$-tuple
of unitaries for the elements of $X$ appearing in a reduced expression
of $w$. For example if $w=x_{1}x_{2}x_{1}^{-1}x_{2}^{-1}$ then $w(u_{1},u_{2},\ldots,u_{r})=u_{1}u_{2}u_{1}^{-1}u_{2}^{-1}$.
Let $|w|$ denote the word length of $w$ w.r.t. $X$. Let 
\begin{equation}
\E_{n}[s_{\lambda,\mu}(w)]\eqdf\int_{\U(n)^{r}}s_{\lambda,\mu}(w(u_{1},u_{2},\ldots,u_{r}))du_{1}\cdots du_{r}.\label{eq:exp-trace-word-def}
\end{equation}
For $L\in\N$ let 
\begin{equation}
g_{L}(x)\eqdf\prod_{c=1}^{L}(1-c^{2}x^{2})^{\lfloor\frac{L}{c}\rfloor}.\label{eq:g-def}
\end{equation}

The input into our analysis is the following result about expected
values of stable characters of word maps, building on the works \cite{MageePuder1,MageeRURSGI,MageeRURSGII}.
\begin{thm}
\label{thm:random-matrix-est}Let $w\in\F_{r}$ be not the identity
with $|w|\leq q$.
\begin{enumerate}
\item \label{enu:deg-of-poly}There is a polynomial $P_{\lambda,\mu,w}\in\Q[x]$
such that for\\
 $n\geq(k+\ell)|w|$,
\[
\E_{n}[s_{\lambda,\mu}(w)]=\frac{P_{\lambda,\mu,w}\left(\frac{1}{n}\right)}{g_{(k+\ell)q}\left(\frac{1}{n}\right)},
\]
with 
\[
\deg\left(P_{\lambda,\mu,w}\right)\leq3(k+\ell)q\left(1+\log\left((k+\ell)q\right)\right).
\]
\item \label{enu:generic-decay}If $w$ is not a proper power in $\F_{r}$,
then $\E_{n}[s_{\lambda,\mu}(w)] = O(n^{-\frac 1 6 (k+\ell)})$.
%%   then for 
%% \[
%% \chi>-\frac{(k+\ell)}{3},
%% \]
%% $a_{\lambda,\mu}(w;\chi)=0$. 
\item \label{enu:poly}If $w$ is not a proper power in $\F_{r}$ and $\mu=\emptyset$, then  $\E_{n}[s_{\lambda,\emptyset}(w)] = O(n^{-(k+\ell)})$.
%%   when
%% \[
%% \chi>-k,
%% \]
%% $a_{\lambda,\emptyset}(w,\chi)=0$.
\end{enumerate}
\end{thm}

It is probable that the conclusion of Part \ref{enu:generic-decay}
also holds with $O(n^{-(k+\ell)})$, but we are unsure how to prove
this at the moment, and it just changes constants in our theorems.

Part \ref{enu:poly} of Theorem \ref{thm:random-matrix-est} follows
from combining a result of \cite{MageePuder1} with one of Duncan
and Howie \cite{DuncanHowie} as explained in $\S$\ref{sec:Improved-bounds-for-poly-char}.

Theorem \ref{thm:random-matrix-est} will be proved in the later part
of the paper ($\S$\ref{sec:Transverse-maps}-\ref{sec:Proof-of-Theorem-random-matric}).

In $\S$\ref{sec:Markov-brothers-inequality}-\ref{sec:Proof-of-Theorem-main}
we prove Theorem \ref{thm:main} from Theorem \ref{thm:random-matrix-est}.

\section{Markov brothers inequality\label{sec:Markov-brothers-inequality}}

We use the following fundamental inequality by Andrey and Vladimir
Markov, see \cite[Section 4.1]{chen2024new} for references. If $P$
is a degree $\leq D$ polynomial in one variable, then for every integer
$k$, 
\begin{equation}
\sup_{[-1,1]}|P^{(k)}|\leq\frac{D^{2}(D^{2}-1)\dots(D^{2}-(k-1)^{2})}{(2k-1)!!}\sup_{[-1,1]}|P|\label{eq:Markov1}
\end{equation}
where $(2k-1)!!=1\cdot3\dots(2k-1)$. By an affine change of variable,
for a general interval (\ref{eq:Markov2}) becomes: 
\begin{equation}
\sup_{[a,b]}|P^{(k)}|\leq\frac{2^{k}}{(b-a)^{k}}\frac{D^{2}(D^{2}-1)\dots(D^{2}-(k-1)^{2})}{(2k-1)!!}\sup_{[a,b]}|P|.\label{eq:Markov2}
\end{equation}

We deduce the following, which is essentially \cite[Lemma 4.2]{chen2024new}.
\begin{lem}
\label{lemm:epsilon-net} For every polynomial $P$ of degree $\leq D$
and every integer $N\geq D^{2}$
\[
\sup_{[0,\frac{1}{N}]}|P|\leq\frac{1}{1-\frac{D^{2}}{N+1}}\sup_{n\geq N}\left|P\left(\frac{1}{n}\right)\right|.
\]
\end{lem}

\begin{proof}
Every element of the interval $[0,\frac{1}{N}]$ is at distance at
most $\frac{1}{2N(N+1)}$ from $\{\frac{1}{n}\mid n\geq N\}$. Therefore,
by the fundamental inequality of calculus, we have 
\[
\sup_{[0,\frac{1}{N}]}|P|\leq\sup_{n\geq N}\left|P\left(\frac{1}{n}\right)\right|+\frac{1}{2N(N+1)}\sup_{[0,\frac{1}{N}]}|P'|.
\]
By the Markov brothers inequality (\ref{eq:Markov2}) with $k=1$,
this is less than 
\[
\sup_{n\geq N}\left|P\left(\frac{1}{n}\right)\right|+\frac{1}{2N(N+1)}\cdot2ND^{2}\sup_{[0,\frac{1}{N}]}|P|.
\]
The lemma follows. 
\end{proof}
As a consequence, 
\begin{lem}
\label{lemm:Markov_after_epsilon_net} For every polynomial $P$ of
degree $\leq D$ and every integer $k\leq D$, 
\[
\sup_{[0,\frac{1}{2D^{2}}]}|P^{(k)}|\leq\frac{2^{2k+1}D^{4k}}{(2k-1)!!}\sup_{n\geq D^{2}}|P\left(\frac{1}{n}\right)|.
\]
\end{lem}

\begin{proof}
By (\ref{eq:Markov2}) with $a=0$ and $b=\frac{1}{2D^{2}}$, 
\[
\sup_{[0,\frac{1}{2D^{2}}]}|P^{k}|\leq(4D^{2})^{k}\frac{D^{2}(D^{2}-1)\dots(D^{2}-(k-1)^{2})}{(2k-1)!!}\sup_{[0,\frac{1}{2D^{2}}]}|P|.
\]
The lemma follows from the bound $D^{2}(D^{2}-1)\dots(D^{2}-(k-1)^{2})\leq D^{2k}$
and Lemma~\ref{lemm:epsilon-net}. 
\end{proof}

\section{A criterion for strong convergence\label{sec:Temperdness_criterion}}
In this section, $\Gamma$ will be a finitely generated group with a
fixed finite generating set and corresponding word-length. We denote by 
$\C_{\leq q}[\Gamma]$ the subspace of the elements of the group
algebra of $\Gamma$ supported in the ball of radius $q$. 

\begin{defn}\label{defn:tempered}
  A function $u \colon \Gamma \to \C$ is tempered if
  \[\limsup_{n \to \infty} |u((x^*x)^n)|^{\frac 1 {2n}} \leq \|\lambda(x)\|\] for every $x \in \C[\Gamma]$. Here $\lambda$ is the left-regular representation of $\Gamma$ on $\ell_2(\Gamma)$ and the norm is the operator norm.

  A function $u \colon \Gamma \to \C$ is completely tempered
  if \[\limsup_{n \to \infty} |\Tr \otimes u((x^*x)^n)|^{\frac 1 n}
  \leq \|(\mathrm{id}\otimes \lambda)(x)\|\] for every integer $k \geq
  1$ and every $x \in M_k(\C)\otimes \C[\Gamma]$, where $\Tr$ is the trace on $M_k(\C)$.
\end{defn}
For example, if $u(\gamma) = \langle \pi(\gamma)\xi,\xi\rangle$ is a matrix coefficient of a unitary representation of $\Gamma$, then $u$ is tempered if and only if it is completely tempered, if and only if the cyclic representation generated by $\xi$ is weakly-contained in the left-regular representation. So Definition~\ref{defn:tempered} can be seen as an adaptation to functions that are not related to representations of the classical notion of tempered representation.

This section is devoted to the proof of the following criterion for strong convergence in probability of a sequence of random representations. 
\begin{prop}\label{prop:criterion_for_strong_convergence}
Let $\pi_n$ be a sequence of random unitary representations of finite nonrandom dimension, $u_n \colon \Gamma \to \C$ be functions and $\varepsilon_n>0$. Assume that
  \begin{itemize}
  \item $u_n$ is tempered and there is a polynomial $P_n$ such that for every $q$ and every $x \in \C_{\leq q}[\Gamma]$, $|u_n(x)| \leq P_n(q) \|x\|_{C^*(\Gamma)}$,
  \item $\big|\E \Tr(\pi_n(x)) - u_n(x)\big| \leq \varepsilon_n \exp\left(\frac{q}{\log(2+q)^2}\right) \|x\|_{C^*(\Gamma)}$ for every $q$ and every $x \in \C_{\leq q}[\Gamma]$,
  \item $\lim_n \varepsilon_n = 0$.\end{itemize}

  Then for every $y \in \C[\Gamma]$, and every $\delta>0$,
  \[\lim_n \prob( \|\pi_n(y)\| \geq \|\lambda(y)\|+\delta) = 0.\]

  If $C^*_\lambda(\Gamma)$ has a unique trace, then the conclusion becomes that
  \[\lim_n \prob( \big| \|\pi_n(y)\| - \|\lambda(y)\| \big| \geq \delta) = 0.\]
\end{prop}
\begin{rem}
  The choice of $w(q) = \exp\left(\frac{q}{\log(2+q)^2}\right)$ is an arbitrary choice that is convenient for our applications, the crucial property that is used is that $\sum_q \frac{1}{1+q^2} \log w(q)<\infty$, reminiscent of the Beurling-Malliavin Theorem. It could be replaced by $\exp( q^2 u_q)$ for an arbitrary decreasing and summable sequence $(u_q)_{g \geq 0}$.
  \end{rem}
\begin{rem}
  Proposition~\ref{prop:criterion_for_strong_convergence} is a variant with high derivatives of the criterion for strong convergence that appeared implicitly, for polynomial $w$, in \cite{HaagerupThr} with $u_n(\gamma) = 1_{\gamma=1}$, and \cite{Schultz} for general $u_n$.
\end{rem}
\begin{rem}
  The condition on the existence of $P_n$ is probably not needed, but will be trivially satisfied in the applications. It allows to use standard results on distributions rather than ad-hoc proofs.
\end{rem}

We have the following variant for matrix coefficients. 
\begin{prop}\label{prop:criterion_for_strong_convergence_coefficients} Let $\pi_n,d_n,u_n,\varepsilon_n$ be as in Proposition~\ref{prop:criterion_for_strong_convergence}. Assume moreover that $u_n$ is completely tempered. Let $k_n$ be a sequence of integers such that $\lim_n \varepsilon_n k_n=0$.

  Then for every integer $q$, $\delta>0$ and every sequence $y_n \in \C_{\leq q}[\Gamma] \otimes M_{k_n}(\C)$ with $\|y_n\|_{C^*(\Gamma)\otimes M_{k_n}} \leq 1$,
  \[\lim_n \prob( \|\pi_n(y_n)\| \geq \|\lambda(y_n)\|+\delta) = 0.\]

  If $C^*_\lambda(\Gamma)$ has a unique trace and is exact, then the conclusion becomes that
  \[\lim_n \prob( \big| \|\pi_n(y_n)\| - \|\lambda(y_n)\| \big| \geq \delta) = 0.\]
\end{prop}

We first collect a few ingredients that we need for the proof.

\subsection{Bump functions with small Fourier coefficient}
A nonzero compactly supported continuous function cannot have a Fourier transform that decays too fast at infinity, as it has to satisfy $\int \frac{\log |\hat f(t)|}{1+t^2} dt > -\infty$. In other words, the condition $\int_{0}^\infty  \frac{\varphi(t)}{1+t^2}dt <\infty$ on a function $\varphi \colon \R^+ \to \R^+$ is a necessary condition for the existence of a nonzero compactly supported continuous function $f$ such that $|\hat f(t) | \leq \exp(-\varphi(|t|)$ for every $t$. The Beurling-Malliavin theorem asserts that this is often also a sufficient condition, see for example \cite{MashreghiNazarovHavin}. The following elementary fact is an illustration of the same phenomenon. The small nuance compared to the Beurling-Malliavin theorem is that we need that $f$ be nonnegative. The same proof works for $\varphi(t) = \frac{t}{(\log(2+|t|))^{1+\varepsilon}}$ replaced by any continuous function $[0,\infty) \to [0,\infty)$ such that $\frac{\varphi(t)}{t^2}$ is non-increasing and integrable on $[1,\infty)$.

\begin{lem}
\label{section:bump} For every $\varepsilon>0$, there is a real
number $M>0$ and an even function $C^{\infty}$ function $f\colon\R\to\R$
that is strictly positive on $(-1,1)$, zero outside of $(-1,1)$, and such
that 
\[
|\hat{f}(t)|\leq\exp\left(-M\frac{|t|}{(\log(2+|t|))^{1+\varepsilon}}\right).
\]
\end{lem}

\begin{proof}
The proof is a standard construction of a smooth bump function as
infinite convolution product of indicator functions \cite[Lemma V.2.7]{Katznelson}.
Define $a_{j}=\frac{c}{j(\log2+j)^{1+\varepsilon}}$, where $c>0$
is chosen such that $\sum_{j\geq1}a_{j}=1$. Let $(X_{j})_{j\geq1}$
be a sequence of independent random variables, all uniform in $[-1,1]$.
Let $\mu$ be the law of $\sum_{j}a_{j}X_{j}$. It is a measure with
full support in $[-\sum_{j}a_{j},\sum_{j}a_{j}]=[-1,1]$. Then 
\[
\widehat{\mu}(t)=\E\exp(-it\sum a_{j}X_{j})=\prod_{j\geq1}\frac{\sin(ta_{j})}{ta_{j}}.
\]
The inequality 
\[
|\widehat{\mu}(t)|\leq\prod_{j\leq\frac{t}{(\log2+t)^{1+\varepsilon}}}\frac{1}{|t|a_{j}}\leq\exp\left(-M\frac{|t|}{(\log2+t)^{1+\varepsilon}}\right)
\]
is a computation. It implies that $\mu$ is absolutely continuous
with respect to the Lebesgue measure and that its Radon-Nikodym derivative,
which is our $f$, is $C^{\infty}$. $f$ does not vanish on the interval
$(-1,1)$ because $f$ is a log-concave function as a limit of convolutions
of log-concave functions. 
\end{proof}
\begin{rem}
In this lemma, we have tried to optimize the rate of decay of the
Fourier transform of $f$, because this is what allows to take $A$
the largest in Theorem~1.1. But this is reflected by the fact that
$f(t)$ is extremely small when $t<1$ is close to $1$, something
like $\exp(-\exp(O(\frac{1}{1-t})))$, which gives rather bad quantitative
estimates for the speed of the convergence. For smaller values of
$K\ll n^{A}$, it would better to use functions with slower decay
of $\hat{f}$. 
\end{rem}

We can periodize the above function:
\begin{lem}
\label{lemm:bump} For every $0<\alpha<\pi$, there is a nonnegative
real-valued even function $\varphi_{\alpha}\in\C^{\infty}(\R/2\pi\Z)$ that is strictly positive on $(-\alpha,\alpha) +2\pi \Z$, zero outside of $(-\alpha,\alpha) +2\pi \Z$ 
with support equal to $[\alpha,\alpha]$ and such that 
\[
\left|\hat{\varphi}_{\alpha}(k)\right|\leq\exp\left(-M\frac{|k|\alpha}{\log(2+|k|\alpha|)^{1+\varepsilon}}\right)
\]
for every $k\in\Z$. 
\end{lem}

\begin{proof}
Set $\varphi_{\alpha}(t+2\pi\Z)=f(t/\alpha)$ for every $t\in[-\pi,\pi]$. 
\end{proof}

\subsection{Sobolev algebras in the full group $C^{*}$-algebra}\label{subsec:Sobolev}

Let $(w(q))_{q\in\N}$ be a non-decreasing sequence of positive real
numbers satisfying $w(q_{1}+q_{2})\leq w(q_{1})w(q_{2})$. Set
\[
\mathcal{S}_{w}(\Gamma)=\big\{ x\in C^{*}(\Gamma)\mid x=\sum_{q}x_{q},x_{q}\in\C_{\leq q}[\Gamma],\sum_{q}w(q)\|x_{q}\|_{C^{*}(\Gamma)}<\infty\big\}.
\]
It is a Banach algebra for the norm 
\[
\|x\|_{w}:=\inf\big\{\sum_{q}w(q)\|x_{q}\|_{C^{*}(\Gamma)}\mid x=\sum_{q}x_{q},x_{q}\in\C_{\leq q}[\Gamma]\big\}.
\]
For $w(q)=(1+q)^{i}$ we denote the corresponding space $\mathcal{S}_{i}(\Gamma)$. We have the following elementary fact:
\begin{lem}\label{lem:linear_forms_on_Sw} A function $u \colon \Gamma \to \C$ extends by linearity to a linear map $\mathcal{S}_w(\Gamma)\to \C$ of norm $\leq N$ if and only if for every $q$ and every $x \in \C_{\leq q}[\Gamma]$,
  \[ |u(x):=\sum_\gamma x(\gamma) u(\gamma)| \leq N w(q) \|x\|_{C^*(\Gamma)}.\]
\end{lem}

The next lemma is useful and suggests the terminology of Sobolev algebras for the algebras $\mathcal{S}_i(\Gamma)$.
\begin{lem}\label{lem:C_i+1} Let $y=y^{*}\in\C_{\leq q}[\Gamma]$ and $f:[-\|y\|_{C^{*}(\Gamma)},\|y\|_{C^{*}(\Gamma)}]\to\R$ be a continuous function. Let $w \colon \N \to \R$ as above. Assume that the function $\varphi(\theta) = f(\|y\|_{C^{*}(\Gamma)} \cos \theta)$ satisfies $\sum_{n \in \Z} w(q |n|) |\hat \varphi(n)|<\infty$. Then $f(y)\in\mathcal{S}_{w}(\Gamma)$ and
  \[ \|f(y)\|_w \leq \sum_{n \in \Z} w(q |n|) |\hat \varphi(n)|.\]

In particular, if $f$ is $C^{i+1}$, then $f(y) \in \mathcal{S}_i(\Gamma)$ with norm $\leq C_i (1+\|y\|_{C^*(\Gamma)})^{i+1} (1+q)^i \|f\|_{C^{i+1}}$.
\end{lem}
\begin{proof} In the proof, we write $\|y\|$ for $\|y\|_{C^{*}(\Gamma)}$. The assumption implies that the Fourier coefficients of $\varphi$ are absolutely summable, which gives rise to a norm-converging expansion
\[
f(y)=\sum_{n \in \Z}\hat{\varphi}(n)T_{n}(y/\|y\|)
\]
where $T_n$ is the $n$-th Chebyshev polynomial $T_n(\cos \theta)=\cos(n\theta)$. So $T_{n}(y/\|y\|)$ belongs to $\C_{\leq q|n|}[\Gamma]$, has norm $\leq 1$ in $C^*(\Gamma)$ 
and 
\[
\sum_{n}w(q|n|)|\hat{\varphi}(n)|\|T_{n}(y/\|y\|)\|_{C^{*}(\Gamma)}\leq \sum_{n}w(q|n|)|\hat{\varphi}(n)|.
\]
The proves the first part of the lemma. 

For the second part, if $f$ is $C^{i+1}$, then so is $\varphi$ and by the chain rule we have $\|\varphi\|_{C^{i+1}} \lesssim (1+\|y\|)^{i+1}\|f\|_{C^{i+1}}$ and therefore by the Plancherel formula
\begin{equation}\label{eq:plancherel_phi} (\sum_{n}|\widehat{\varphi^{(i+1)}}(n)|^{2})^{\frac{1}{2}}\lesssim (1+\|y\|)^{i+1}\|f\|_{C^{i+1}}
\end{equation}
with constant that depend on $i$ only. By the Cauchy-Schwarz inequality and the inequality $(1+q|n|) \leq (1+q)(1+|n|)$
\[
\sum_n (1+q|n|)^i |\hat \varphi(n)| \leq (1+q)^i \big(\sum_{n}(1+|n|)^{-2}\big)^{\frac{1}{2}}\big(\sum_{n}(1+|n|)^{2i+2}|\hat{\varphi}(n)|^{2}\big)^{\frac{1}{2}}.
\]
We bound $(1+|n|) \leq 1_{n=0}+2|n|$ and write $n^{i+1} \hat\varphi(n) = \widehat{\varphi^{(i+1)}}(n)$, so that
\[ (\sum_{n}(1+|n|)^{2i+2}|\hat{\varphi}(n)|^{2})^{\frac{1}{2}} \leq |\hat{\varphi}(0)| + 2^{i+1}(\sum_{n}|\widehat{\varphi^{(i+1)}}(n)|^{2})^{\frac{1}{2}},\]
The lemma follows by \eqref{eq:plancherel_phi}.
\end{proof}
\begin{prop}\label{prop:characterization_of_temperedness_in_Sobolev_algebras} Let $u \colon \Gamma \to \C$. Assume that $u$ extends to a continuous linear map $u\colon \mathcal{S}_i(\Gamma) \to \C$. Then $u$ is tempered if and only if there is $n$ such that for every selfadjoint $x \in \mathcal{S}_i(\Gamma) \cap \ker \lambda$, $u(x^n)=0$. In that case, this holds for all $n\geq i+2$.

Furthermore, $u$ is completely tempered if and only if there is $n$ such 
that for every $k$ and every selfadjoint $x \in M_k(\C)\otimes 
(\mathcal{S}_i(\Gamma) \cap \ker \lambda)$, $(\mathrm{Tr}\otimes 
u)(x^n)=0$. In that case, this holds for all $n\geq i+2$.
\end{prop}
\begin{proof} We prove the equivalence for temperedness, the proof for complete temperedness is identical.

  We start with a preliminary observation. Lemma~\ref{lem:C_i+1} tells us that for every self-adjoint $y \in \C_{\leq q}(\Gamma)$, the function $f \in C^\infty(\R) \mapsto u(f(y))$ is a compactly supported distribution of order $\leq i+1$, and the largest symmetric interval containing its support is $[-\limsup|u(y^k)|^{\frac 1 k},\limsup|u(y^k)|^{\frac 1 k}]$, see \cite[Lemma 4.9]{chen2024new}).

 Given this observation, the if direction is direct. Let $\varepsilon>0$, and pick $f$ be a $C^\infty$ function equal to $0$ on the interval $[-\|\lambda(y)\|^2,\|\lambda(y)\|^2]$ and to $1$ outside of the $\varepsilon$-neighbourhood of it.  For $k$ large enough (so that $t\mapsto |t|^{k/n}$ is $C^{i+2}$), $x=(y^*y)^{k/n} f(y^*y)$ belongs to $\mathcal{S}_i(\Gamma) \cap \ker \lambda$. As a consequence, $u(x^{n})=0$ and therefore
  \[ u( (y^*y)^{k}) = u( (y^*y)^{k}(1-f(y^*y)^n)).\]
  But the $C^{i+1}$ norm of $t\mapsto t^{k}(1-f(t)^n)$ is $\leq C(i,\varepsilon,y) (\|\lambda(y)\|^2+\varepsilon)^{k}$, which implies that
  \[ \limsup_k |u( (y^*y)^k)|^{1/2k} \leq \sqrt{\|\lambda(y)\|^2+\varepsilon}.\]
  This proves that $u$ is tempered.

  For the converse, consider a self-adjoint $x \in \mathcal{S}_i(\Gamma)\cap \ker\lambda$. Decompose $x= \sum_q x_q$ as in the definition of $\mathcal{S}_i(\Gamma)$. We can moreover assume that each $x_q$ is self-adjoint. Let $y_q = \sum_{s\leq q} x_s \in \C_{\leq q}[\Gamma]$. Then we have
  \begin{equation}\label{eq:bound_on_lambdayq} \|\lambda(y_q)\| = \|\lambda(y_q-x)\|\leq \sum_{s>q} \|x_s\| = o((1+q)^{-i}),
  \end{equation}
  because $\sum_q (1+q)^i\|x_q\|<\infty$. Let $\varphi$ be a $C^\infty$ function equal to $1$ on $[-1,1]$ and $0$ outside of $[-2,2]$, and let $\varphi_q(t) = \varphi(\frac{t}{\|\lambda(y_q)\|})$. We know from the preliminary observation that $\Lambda_q(f) = u(f(y_q))$ is a distribution with support inside $[-\|\lambda(y_q)\|,\|\lambda(y_q)\|]$ and such that $|\Lambda_q(f) |\leq C (1+q)^i \|f\|_{C^{i+1}(\R)}$ for some $C$ independent from $q$. Therefore, for every $n$, since the functions $t\mapsto t^n$ and $f:t \mapsto t^n \varphi_q(t)$ coincide on a neighbourhood of $[-\|\lambda(y_q)\|,\|\lambda(y_q)\|]$, we have
  \begin{equation}\label{eq:bound_on_u_y_qn} u(y_q^n) = u(y_q^n \varphi_q(y_q))=(1+q)^i O_{q \to \infty}(\|\lambda(y_q)\|^{n-i-1}).
  \end{equation}
  The last inequality is because the function $f$ has $C^{i+1}$-norm $\leq C(i,n) \|\lambda(y_q)\|^{n-i-1}$. Indeed, for $k \leq i+1$,
  \[ f^{(k)}(t) = \sum_{j=0}^k \binom{k}{j} \frac{n!}{(n-j)!} \frac{t^{n-j}}{\|\lambda(y_q)\|^{k-j}} \varphi^{(k-j)}\Big(\frac{t}{\|\lambda(y_q)\|}\Big).\]
  Using  $|t| \leq 2 \|\lambda(y_q)\|$ on the support of $\varphi_q$, we see that this is less than $2^n(1+n)^k\|\lambda(y_q)\|^{n-k} \|\varphi\|_{C^k}$ as claimed. If $n\geq i-2$ we conclude from  \eqref{eq:bound_on_lambdayq} and \eqref{eq:bound_on_u_y_qn} that $\lim_q u(y_q^n)=0$. Therefore, we have
    \[ u(x^n) = u(\lim_q y_q^n) = \lim_q u(y_q^n)=0.\]
The proposition is proved. 
\end{proof}
\subsection{Unique trace and exactness}
We use the following easy consequence of the operator-space reformulation of exactness, in the line of \cite[Theorem 7.3]{BordenaveCollins3}.
\begin{lem}\label{lem:exactness_in_terms_of_fd_approximations} Let $A$ be an exact $C^*$-algebra. For every finite family $a_1,\dots,a_n$ in $A$ and every $\varepsilon>0$, there is an integer $d$ such that for every $k$ and every $b_1,\dots,b_n \in M_k(\C)$, there are norm $\leq 1$ matrices $u,v \in M_{d,k}(\C)$ such that
  \[  (1-\varepsilon)\| \sum_i a_i \otimes b_i\| \leq \|\sum_i a_i \otimes u b_{i} v^*\|.\]
\end{lem}
\begin{proof}
If $a_i$ belong to $M_N(\C)$, the lemma is easy with $d=N$, because the 
norm of $X=\sum a_i \otimes b_i$ is the supremum of $\langle X 
\xi,\eta\rangle$ over norm $1$ elements of  $\C^N \otimes \C^k$, and 
every element of $\C^N \otimes \C^k$ belongs to $\C^N \otimes E$ for a 
subspace $E \subset \C^k$ of dimension $\leq N$. The general case 
follows by Kirchberg's characterization of exact $C^*$-algebras, which 
implies that for every $\varepsilon>0$, there is $N$ such that the 
operator space spanned by $a_1,\dots,a_n$ is at completely bounded 
distance $\leq 1+\varepsilon$ from a subspace of $M_N(\C)$ 
\cite[Corollary 17.5]{PisierOS}. This means that there exist $A_1,\dots,A_n \in M_N(\C)$ such that
\[ \| \sum_{i=1}^n a_i \otimes b_i\| \leq \| \sum_{i=1}^n A_i \otimes b_i\|  \leq (1+\varepsilon)\| \sum_{i=1}^n a_i \otimes b_i\|\]
for every $k$ and every $b_1,\dots,b_n \in M_k(\C)$.
\end{proof}
\begin{lem}\label{lem:consequence_of_unique_trace} Let $\Gamma$ be a finitely generated group. Assume that $C^*_\lambda(\Gamma)$ has a unique trace. Then for every integer $q$ and $\varepsilon>0$, there exist $y_1,\dots,y_n \in \C[\Gamma]$ and $\delta>0$ such that, for every finite dimensional unitary representation $\pi$ of $\Gamma$, if
  \[ \max_i \frac{\|\pi(y_i)\|}{\|\lambda(y_i)\|} \leq 1+\delta\]
  then
  \begin{equation}\label{eq:consequence_of_unique_trace} \inf_{y \in \C_{\leq q}[\Gamma]} \frac{\|\pi(y)\|}{\|\lambda(y)\|} \geq 1-\varepsilon.
  \end{equation}
If moreover $C^*_\lambda(\Gamma)$ is exact, the conclusion \eqref{eq:consequence_of_unique_trace} can be strengthened to  \[ \inf_{{\substack{d\geq 1\\ y \in M_d(\C)\otimes \C_{\leq q}[\Gamma]}}} \frac{\|(\mathrm{id}\otimes \pi)(y)\|}{\|(\mathrm{id}\otimes \lambda)(y)\|} \geq 1-\varepsilon.\]
\end{lem}
\begin{proof}We use the following fact: $C^*_\lambda(\Gamma)$ has a unique trace if 
and only if for every $\gamma \in \Gamma \setminus\{1\}$ and every 
$\varepsilon$, there is $\delta>0$ and a finite family $y_1,\dots,y_k 
\in \C[\Gamma]$ such that, for every unitary representation $\pi\colon 
\Gamma \to A$ to a $C^*$-algebra $A$ with a tracial state $\sigma$, 
$\max_i \frac{\|\pi(y_i)\|}{\|\lambda(y_i)\|} \leq 1+\delta$ implies 
$|\sigma(\pi(\gamma))| \leq\varepsilon$. The only if direction is 
direct: if $\sigma$ is an arbitrary tracial state on 
$C^*_\lambda(\Gamma)$, by applying the criterion to 
$A=C^*_\lambda(\Gamma)$ and $\pi$ the left-regular representation, we 
obtain that $|\sigma(\lambda(\gamma))|\leq \varepsilon$ for every 
$\gamma \neq 1$ and $\varepsilon>0$, that is $\sigma$ is the standard 
trace. The converse is proved by standard ultraproduct arguments (see 
\cite[Lemma 6.1]{louder2023strongly} for the details). %% Assume that the criterion does not hold. This means that there exists $\gamma \in \Gamma \setminus\{1\}$, $\varepsilon>0$ and a sequence\footnote{Here we use that $\Gamma$ is countable; otherwise the statement is true with a similar proof, using nets rather than sequences} of $C^*$-algebras $A_\alpha$ with tracial states $\sigma_\alpha$ and unitary representations $\pi_\alpha\colon \Gamma\to A_\alpha$, such that for every $y \in \C[\Gamma]$, $\limsup_\alpha \|\pi_\alpha(y)\| \leq \|\lambda(y)\|$ but $|\sigma_\alpha(\gamma)| >\varepsilon$ for every $\alpha$. Let $\mathcal{U}$ be a non-principal ultrafilter on the set of integers, and define $\sigma_{\mathcal{U}}(y) = \lim_{\mathcal{U}} \sigma_\alpha(\pi_\alpha(y))$ for every $y \in \C[\Gamma]$. We have $|\sigma_\alpha(\pi_\alpha(y))| \leq \|\pi_\alpha(y)\|$ for every $\alpha$, and therefore
%% \[ |\sigma_{\mathcal{U}}(y)| \leq \limsup_\alpha \|\pi_\alpha(y)\| \leq \|\lambda(y)\|.\]
%% So $\sigma_{\mathcal{U}}$ extends by continuity to a tracial state on $C^*_\lambda(\Gamma)$, which is not the standard state because $|\sigma_{\mathcal{U}}(\lambda(\gamma))| \geq \varepsilon$. 

The first conclusion of the lemma follows quite easily from this 
characterization of the unique trace property. Indeed, if $y \in 
\C[\Gamma]$, there is an integer $p$ such that $\|\lambda(y)\| \leq 
(1+\varepsilon) \tau((y^*y)^p)^{\frac 1 {2p}}$, and by compactness the 
same $p$ can be taken for every $y \in \C_{\leq q}[\Gamma]$. Now 
applying the criterion for every $\gamma$ in the ball of radius $2qp$, 
we obtain that there is a finite family $y_1,\dots,y_k \in \C[\Gamma]$ 
such that, for $\pi,A,\sigma$ as above, $\max_i 
\frac{\|\pi(y_i)\|}{\|\lambda(y_i)\|} \leq 1+\delta $ implies that for 
every $y \in \C_{\leq q}[\Gamma]$, $\tau((y^*y)^p)^{\frac 1 {2p}} \leq 
(1+\varepsilon)\sigma(\pi((y^*y)^p))^{\frac 1 {2p}}$, and in particular 
$\|\lambda(y)\| \leq (1+\varepsilon)^2 \|\pi(y)\|$. This is 
\eqref{eq:consequence_of_unique_trace}, up to a change of $\varepsilon$.

Observe that exactly the same argument also shows that, without any 
further assumption on $C^*_\lambda(\Gamma)$, for any fixed $d$, 
\eqref{eq:consequence_of_unique_trace} can be replaced by
   \[ \inf_{y \in M_d(\C) \otimes \C_{\leq q}[\Gamma]} 
\frac{\|(\mathrm{id}\otimes \pi)(y)\|}{\|(\mathrm{id}\otimes 
\lambda)(y)\|} \geq 1-\varepsilon.\]
   But of course, $y_i$ and $\delta$ a priori depend on $d$. If 
$C^*_\lambda(\Gamma)$ is exact, 
Lemma~\ref{lem:exactness_in_terms_of_fd_approximations} allows to remove 
this dependence and to conclude the proof of the lemma.

\end{proof}
\subsection{Proof of Proposition~\ref{prop:criterion_for_strong_convergence} and Proposition~\ref{prop:criterion_for_strong_convergence_coefficients}}

\begin{proof}[Proof of Proposition~\ref{prop:criterion_for_strong_convergence}]
  Let $i_n$ be the degree of $P_n$ and $w(q) = \exp(\frac{q}{\log (2+q)^2})$. Lemma~\ref{lem:linear_forms_on_Sw} implies that $u_n$ extends by continuity to $\mathcal{S}_i(\Gamma)$, and that on its subspace $\mathcal{S}_w(\Gamma)$,
  \begin{equation}\label{eq:master_inequality} \big|\E \Tr(\pi_n(x)) - u_n(x)\big| \leq \varepsilon_n \|x\|_{w}.
  \end{equation}
  
  Without loss of generality, we can take $y\neq 0$ and normalize it so that $\|y\|_{C^{*}(\Gamma)}=1$. 
%If $\alpha=0$, that is $\|\lambda(y)\|=1=\|y\|_{C^{*}(\Gamma)}$, the conclusion is clear. So let us assume
%that $\alpha>0$.

By Lemma~\ref{lem:C_i+1}, the map $f \mapsto u_n(f(y^*y/2))$, is continuous for the\\ 
$C^{i_n+1}([-1/2,1/2])$ topology, so it extends to a distribution. The assumption that $u_n$ is tempered in turn implies
that its support in contained in
$[-\|\lambda(y)\|^2/2,\|\lambda(y)\|^2/2]$ (see \cite[Lemma 4.9]{chen2024new}). In particular, using $\|\lambda(y)\|>0$, we have that $u_n(f(y^*y))=0$ for every $f$ that vanishes on $[-\|\lambda(y)\|^2/2,\|\lambda(y)\|^2/2]$ \cite[Theorem 2.3.3]{Hormander}.
  
Let $\alpha=\arccos(\|\lambda(y^{*}y)\|/2)\in[\frac{\pi}{3},\frac{\pi}{2})$. Let $\varphi_{\alpha}$ be the function given by
Lemma~\ref{lemm:bump}, but for $\varepsilon = \frac{1}{2}$. It is even,
so it is of the form $\varphi_{\alpha}(\theta)=f(\cos(\theta))$ for
a continuous function $f\colon[-1,1]\to\R$ that is zero on $[-\|\lambda(y)\|^2/2,\|\lambda(y)\|^2/2]$, strictly positive outside of this interval, and $C^\infty$ on $(-1,1)$. We take $x=f(y^{*}y)$. By the preceding discussion, $u_n(x)=0$ for every $n$. Moreover, by Lemma~\ref{lem:C_i+1}, $x$ belongs to $\mathcal{S}_w(\Gamma)$, because
\[
\sum_{N\geq1}\exp(\frac{qN}{\log(1+qN)^{2}}-M\frac{\alpha N}{(\log(1+\alpha N))^{1+\frac{1}{2}}})<\infty.
\]

So the inequality \eqref{eq:master_inequality} together with the obvious inequality $\|\pi_n(x)\| \leq \Tr \pi_n(x)$ gives
\[ \E \|\pi_n(x)\| \leq \varepsilon_n \|x\|_w,\]
which goes to zero by the last assumption.

Now, since $f$ is strictly positive on $[-1,1]\setminus [-\|\lambda(y)^2\|/2,-\|\lambda(y)^2\|/2]$, there is $c>0$ such that $f(t/2)>c$ if $|t|\geq (\|\lambda(y)\|+\delta)^2$. In particular, we have
\[ \prob( \|\pi_n(y)\|>\|\lambda(y)\|+\delta) \leq \prob(\|\pi_n(x)\| >c) \leq \frac{1}{c} \varepsilon_n \|x\|_w,\]
which goes to $0$. This proves the first part of the proposition.

The second half follows from the first half of Lemma~\ref{lem:consequence_of_unique_trace}.
\end{proof}

\begin{proof}[Proof of Proposition~\ref{prop:criterion_for_strong_convergence_coefficients}] The first part is identical to that of Proposition~\ref{prop:criterion_for_strong_convergence}, considering the distribution $f\mapsto \Tr \otimes u_n(f(y_n^* y_n))$.

  The second part follows, using the full statement of Lemma~\ref{lem:consequence_of_unique_trace}. 
  \end{proof}

\section{Random walks on free groups}

\label{sec:RW} The content of this section will play a key rôle in
the proof of Theorem~\ref{thm:main} (see Lemma~\ref{lem:derivatives_tempered}).

\subsection{Proper powers}
Let $\mu$ be a symmetric probability measure on $\F_{r}$, whose
support is finite, contains the identity element and generates $\F_{r}$.
To save space, we will call such a measure \textbf{reasonable}. Let
$(g_{n})_{n\geq0}$ be the corresponding random walk on $\F_{r}$,
that is $g_{n}=s_{1}s_{2}\dots s_{n}$ for iid $s_{i}$ distributed
as $\mu$. Let $\rho=\rho(\mu)$ be the spectral radius, that is the
norm of $\lambda(\mu)$ on $\ell_{2}(\F_{r})$. We have the easy bound
\begin{equation}
\forall g\in\F_{r},\prob(g_{n}=g)=\langle\lambda(\mu)^{n}\delta_{e},\delta_{g}\rangle\leq\rho^{n}.\label{eq:prob_for_RW_in_terms_of_spectral_radius}
\end{equation}

We say that an element of $\F_{r}$ is a proper power if it is of
the form $h^{d}$ for $h\in\F_{r}$ and $d\geq2$. 
\begin{prop}
\label{prop:properPower} There is a constant $C=C(\mu)$ such that
\[
\prob(g_{n}\textrm{ is a proper power })\leq Cn^{5}\rho^{n}.
\]
\end{prop}

The proof uses that the Cayley graph of $\F_{r}$ is a tree, in the
following form:
\begin{lem}
\label{lem:geodesic} There is a constant $C$ and an integer $a$
such that for every $h\in\F_{r}$ and every $h'$ on the segment from
the identity to $h$,
\[
\prob(g_{n}=h)\leq C\sum_{k=0}^{n+a}\prob(g_{n+a}=h\textrm{ and }g_{k}=h'\}).
\]
\end{lem}

\begin{proof}
When $\mu$ is supported on the generators, this is clear with $C=1$
and $a=0$: any nearest-neighbor path from $e$ to $h$ has to pass
through $h'$. In the general case, the idea is that such a path has
to pass not too far from $h'$, and forcing it to pass exactly through
$h'$ does not cost too much, neither in time nor in probability.
Let us make this idea precise. By our assumptions on $\mu$, there
is an integer $r$ such that the support of $\mu$ is contained in
the ball of radius $r$, and another integer $a_{0}$ such that $c:=\inf_{g\in B(e,r)}\prob(g_{a_{0}}=g)>0$.

Let $T$ be the first hitting time of the ball of radius $r$ around
$h'$, so that if $g_{n}=h$ then $T\leq n$. The lemma will follow
from the following observation: if we add $2a_{0}$ steps to the random
walk after time $T$, with probability at least $c^{2}$ we will be
at $h'$ at time $T+a_{0}$ and back to $g_{T}$ at time $T+2a_{0}$,
and then we can run the random walk as before. More formally, define
another realization of the random walk as follows: let $g'_{i}$ be
an independent copy of the random walk, and define 
\[
\tilde{g}_{n}=\begin{cases}
g_{n} & \textrm{ if }n\leq T\\
g_{T}g'_{n-T} & \textrm{ if }T\leq n\leq T+2a_{0}\\
g_{T}g'_{2a_{0}}g_{T}^{-1}g_{n-2a_{0}} & \textrm{ if }T+2a_{0}\leq n.
\end{cases}
\]
By the Markov property ($T$ is a stopping time) $(\tilde{g}_{n})_{n\geq0}$
is distributed as the random walk with step distribution $\mu$.

Conditionally to $T$, the event $A=\{g'_{2a_{0}}=e\textrm{ and }g'_{a_{0}}=g_{T}^{-1}h'\}$
happens with probability $\geq c^{2}$, and when it happens we have
$\tilde{g}_{n+2a_{0}}=g_{n}$ for every $n\geq T$ and $\tilde{g}_{T+a_{0}}=h'$.
Therefore, we can bound the probability of the event 
\[
B=\big\{\tilde{g}_{n+2a_{0}}=h\textrm{ and }h'\in\{\tilde{g_{k}},0\leq k\leq n+2a_{0}\}\big\}
\]
as follows 
\[
\prob(B)\geq\prob(g_{n}=h\textrm{ and }A)\geq c^{2}\prob(g_{n}=h).
\]
By the union bound we obtain 
\[
c^{2}\prob(g_{n}=h)\leq\sum_{k=0}^{n+2a_{0}}\prob(\tilde{g}_{n+2a_{0}}=h\textrm{ and }\tilde{g_{k}}=h'),
\]
which is the content of the lemma with $C=c^{-2}$ and $a=2a_{0}$. 
\end{proof}
\begin{proof}[\foreignlanguage{british}{Proof of Proposition~\ref{prop:properPower}.}]
Denote by $p_{n}$ the probability that $g_{n}$ is a cyclically
reduced proper power, that is $g_{n}=h^{d}$ for a cyclically reduced
$h$ and $d\geq2$. We shall prove the following two inequalities,
from which the proposition follows immediately: there are constants
$C_{1},C_{2}$ such that 
\begin{equation}
\prob(g_{n}\textrm{ is a proper power })\leq C_{2}n^{2}\max_{k\leq n+2a}p_{n+2a-k}\rho^{k}.\label{eq:dealing_with_not_cyclically_reduced}
\end{equation}

\begin{equation}
p_{n}\leq C_{1}n^{3}\rho^{n},\label{eq:cyclically_reduced}
\end{equation}

We start with the proof of (\ref{eq:cyclically_reduced}). Let $h$
be cyclically reduced and $d\geq2$. The assumption that $h$ is cyclically
reduced means that $h$ belongs to the segment between $1$ and $h^{d}$,
and that $h^{d-1}$ belongs to the segment between $h$ and $h^{d}$
(tautologically if $d=2$). By two applications of Lemma~\ref{lem:geodesic},
we have 
\begin{align*}
\prob(g_{n}=h^{d}) & \leq C^{2}\sum_{k_{1}+k_{2}+k_{3}=n+2a}\prob(g_{k_{1}}=h,g_{k_{1}+k_{2}}=h^{d-1},g_{k_{1}+k_{2}+k_{3}}=h^{d})\\
 & =C^{2}\sum_{k_{1}+k_{2}+k_{3}=n+2a}\prob(g_{k_{1}}=h)\prob(g_{k_{2}}=h^{d-2})\prob(g_{k_{3}}=h).
\end{align*}
The inequality~(\ref{eq:prob_for_RW_in_terms_of_spectral_radius})
allows us to bound the middle term by $\rho^{k_{2}}$, and by symmetry
we have 
\[
\prob(g_{k_{1}}=h)\prob(g_{k_{3}}=h)=\prob(g_{k_{1}}=h)\prob(g_{k_{3}}=h^{-1})=\prob(g_{k_{1}}=h\textrm{ and }g_{k_{1}+k_{2}}=e).
\]
Summing over all cyclically reduced words, we get 
\[
\sum_{h\textrm{ cyclically reduced}}\prob(g_{n}=h^{d})\leq C^{2}\sum_{k_{1}+k_{2}+k_{3}=n+2a}\rho^{k_{2}}\prob(g_{k_{1}+k_{3}}=e).
\]
By (\ref{eq:prob_for_RW_in_terms_of_spectral_radius}), this is less
than $C^{2}(n+2a+1)^{2}\rho^{n+2a}$. Summing over all $d\leq n$
we obtain (\ref{eq:cyclically_reduced}).

We now move to (\ref{eq:dealing_with_not_cyclically_reduced}). Let
$h\in\F_{r}$, not necessarily cyclically reduced. Write $h=w\tilde{h}w^{-1}$
its reduced expression with $\tilde{h}$ cyclically reduced. Then
$h^{d}=w\tilde{h}^{d}w^{-1}$ is the reduced expression of $h^{d}$.
In particular $w$ and $w\tilde{h}^{d}$ are on the segment from the
identity to $h^{d}$, and by two applications of Lemma~\ref{lem:geodesic}
we obtain 
\[
\prob(g_{n}=h^{d})\leq C^{2}\sum_{k_{1}+k_{2}+k_{3}\leq n+2a}\prob(g_{k_{1}}=w)\prob(g_{k_{2}}=\tilde{h}^{d})\prob(g_{k_{2}}=w^{-1}).
\]
If we first sum over all $\tilde{h}$ such that $w\tilde{h}w^{-1}$
is reduced , and then over $w$, we obtain 
\[
\prob(g_{n}\textrm{ is a proper power})\leq C^{2}\sum_{k_{1}+k_{2}+k_{2}\leq n+2a}p_{k_{2}}\sum_{w}\prob(g_{k_{1}}=w)\prob(g_{k_{3}}=w^{-1}).
\]
The formula (\ref{eq:dealing_with_not_cyclically_reduced}) follows
because $\sum_{w}\prob(g_{k_{1}}=w)\prob(g_{k_{3}}=w^{-1})=\prob(g_{k_{1}+k_{3}}=e)\leq\rho^{k_{1}+k_{2}}$,
by (\ref{eq:prob_for_RW_in_terms_of_spectral_radius}). 
\end{proof}

\subsection{Random walks and tempered functions}
The following is a useful criterion for a function on a group with the rapid decay property to be tempered. Recall a group is said to have the rapid decay property if it admits a finite generating set and a polynomial $P$ such that, for every $a\in\C_{\leq q}[\Gamma]$,
\begin{equation}\label{eq:RD}
\|\lambda(a)\|\leq P(q) (\sum_\gamma |a(\gamma)|^2)^{\frac 1 2}.
\end{equation}
Haagerup's inequality \cite[Lemma 1.5]{Haagerup} asserts that free groups with their standard generating sets have the rapid decay property, with $P(R)=3(1+R^2)$.
\begin{prop}\label{prop:criterion_for_tempered_inRDgroups} Let $\Gamma$ be a finitely generated group with the rapid decay property, and $u \colon \Gamma \to \C$ a function. Assume that for every reasonable probability measure $\mu$ on $\Gamma$, if $(\gamma_n)_{n \geq 0}$ is the associated random walk on $\Gamma$,
  \begin{equation}\label{eq:absolutely_tempered} \limsup_n \big(\E |u(\gamma_n)|\big)^{\frac 1 n} \leq \rho(\mu).
  \end{equation}
  Then $u$ is tempered, and even completely tempered.
\end{prop}
\begin{proof} Before we start, let us observe that the assumption \eqref{eq:absolutely_tempered} in fact holds for every symmetric finitely supported probability measure $\mu$, even if it not reasonable in the sense that its support is not generating or does not contain the identity. Indeed, let $\nu$ be a fixed reasonable probability measure. Then for every $\varepsilon>0$, $\mu_\varepsilon = \frac{1}{1+\varepsilon}(\mu+\varepsilon \nu)$ is a reasonable probability measure, so that \eqref{eq:absolutely_tempered} holds for $\mu_\varepsilon$. If $\E_\varepsilon$ and $\E$ denote the law of the random walk with transition $\mu_\varepsilon$ and $\mu$ respectively, the inequality $\mu \leq (1+\varepsilon) \mu_\varepsilon$ gives
  \[ \E |u(\gamma_n)| \leq (1+\varepsilon)^n \E_\varepsilon |u(\gamma_n)|,\]
  so we obtain 
\[\limsup_n \big(\E |u(\gamma_n)|\big)^{\frac 1 n} \leq (1+\varepsilon)\rho(\mu_\varepsilon) \leq \rho(\mu) + \varepsilon \rho(\nu).\]
This proves that \eqref{eq:absolutely_tempered} holds for $\mu$ by taking $\varepsilon \to 0$.

We now move to the proof of the proposition. We claim that for every $k, q \geq 1$ and every $x=x^* \in M_k(\C) \otimes \C_{\leq q}[\Gamma]$,
  \begin{equation}\label{eq:intermediate_to_completely_tempered} \limsup_{n \to \infty} |\Tr \otimes u(x^n)|^{\frac 1 {n}} \leq P(q) \sqrt{k} \|(\mathrm{id}\otimes \lambda)(x)\|.
  \end{equation}
  Here $P$ is the polynomial satisfying \eqref{eq:RD}, given by the assumption that $\Gamma$ has the rapid decay property.
  
By a standard \emph{tensor-power} trick, \eqref{eq:intermediate_to_completely_tempered} implies that $u$ is completely tempered. Indeed, given an arbitrary $x \in M_k(\C) \otimes \C_{\leq q}[\Gamma]$, applying \eqref{eq:intermediate_to_completely_tempered} to $y=(x^*x)^m$ (which belongs to $M_k(\C) \otimes \C_{\leq 2m q}[\Gamma]$ and satisfies $\|(\mathrm{id}\otimes \lambda)(y)\| = \|(\mathrm{id}\otimes \lambda)(x)\|^{2m}$) and raising to the power $\frac 1 {2m}$, it implies
\[ \limsup_{n \to \infty} |\Tr \otimes u((x^*x)^n)|^{\frac 1 {2n}} \leq (P(2mq) \sqrt{k})^{\frac{1}{2m}} \|(\mathrm{id}\otimes \lambda)(x)\|,\]
and in the $m \to \infty$ limit 
\[ \limsup_{n \to \infty} |\Tr \otimes u((x^*x)^n)|^{\frac 1 {2n}} \leq \|(\mathrm{id}\otimes \lambda)(x)\|.\]

So it remains to prove~\eqref{eq:intermediate_to_completely_tempered}. We can normalize $x$ so that $\sum_\gamma \|x(\gamma)\|=1$, where $\|\cdot\|$ is the operator norm on $M_k(\C)$. Let $\mu(s) = \|x(s)\|$; it is a symmetric finitely supported probability measure on $\Gamma$. By the triangle inequality and the inequality $|\Tr(x(s_1)\dots x(s_{n}))| \leq k \|x(s_1)\| \dots \|x(s_n)\|$, we have %$|\Tr(a_1\dots a_{n})| \leq k \|a_1\| \dots \|a_n\|$ for $a_1,\dots,a_n \in M_k(\C)$, we have
\[ |\Tr \otimes u(x^n)| \leq k \sum_{s_1,\dots,s_n \in \Gamma} \mu(s_1)\dots \mu(s_n) |u(s_1\dots s_n)|  = k \E |u(\gamma_n)|.\]
By the assumption in the extended form above, we obtain
\[ \limsup_{n \to \infty} |\Tr \otimes u(x^n)|^{\frac 1 n} \leq \rho(\mu).\]
By the rapid decay assumption \eqref{eq:RD}, this is less than
\[ P(q) (\sum_s \mu(s)^2)^{\frac 1 2}.\]
By by the basic fact $\mu(s)^2=\|x(s)\|^2 \leq \Tr(x(s)^* x(s))$, we have
\[ \sum_s \mu(s)^2 \leq (\Tr \otimes \lambda)(x^*x) \leq k \|(\mathrm{id}\otimes \lambda)(x)\|^2,\]which proves the claimed inequality \eqref{eq:intermediate_to_completely_tempered}.
\end{proof}

\section{Proof of Theorem \ref{thm:main}\label{sec:Proof-of-Theorem-main}
from Theorem \ref{thm:random-matrix-est}}

Let $K\geq1$ be an integer. Most of the objects will depend on $K$
but our notation will not reflect this. The only exception is the
letter $C$, which will always denote a constant that does not depend
on anything, but that can change from one line to the next.

\emph{The aim of this section is to explain why Theorem \ref{thm:random-matrix-est}
implies the strong convergence result from Theorem \ref{thm:main}.}

Let 
\[
\sigma_{n}\eqdf\bigoplus_{|\lambda|+|\mu|=K}\pi_{\lambda,\mu}.
\]
It is a sub-representation of $\bigoplus_{k=0}^{\ell}\pi_{k,K-k}^{0}$,
therefore 
\begin{equation}
\mathrm{dim}(\sigma_{n})\leq(K+1)n^{K}.\label{eq:dimension_of_sigma_n}
\end{equation}
Recall the definition of $g_{L}$ from $\S$\ref{sec:Matrix-integral-results}.
Theorem \ref{thm:random-matrix-est} immediately implies the following
result about $\sigma_{n}$. Throughout this section and as in \S
\ref{sec:Matrix-integral-results} $\E_{n}$ or simply $\E$ will denote
the integral with respect to the Haar measure on $\U(n)^{r}$.
\begin{thm}
\label{thm:summary_of_characters} For every $w\in\F_{r}$, there
is a rational function $\varphi_{w}\in\Q(x)$ such that 
\begin{enumerate}
\item \label{item:def_of_phiw} For every integer $n\geq K\max(|w|,1)$,
\[
\varphi_{w}\left(\frac{1}{n}\right)=\frac{1}{n^{K}}\E_{n}[\Tr(\sigma_{n}(u_{1},\dots,u_{r}))].
\]
\item \label{item:denominator+degree} If $w\neq e$ has length $\leq q$,
then $g_{Kq}\varphi_{w}$ is a polynomial of degree $\leq D_{q}=3Kq\log(1+Kq)$.
\item \label{item:derivatives_not_proper_powers} If $w$ is not a proper
power, then $\varphi_{w}^{(i)}(0)=0$ for all $i<K+\frac{K}{6}$. 
\end{enumerate}
\end{thm}

We collect in the next lemma some of the basic properties of $g_{L}$
that we need. 
\begin{lem}
\label{lem:bounds_on_g_L}Let $L\geq1$ be an integer. The polynomial\\
$g_{L}(t)=\prod_{c=1}^{L}(1-c^{2}t^{2})^{\lfloor\frac{L}{c}\rfloor}$
has the following properties, for every $0\leq t\leq\frac{1}{2L^{2}}$
and every integer $i\geq0$: 
\begin{equation}
\frac{1}{2}\leq g_{L}(t)\leq1,\label{eq:bound_on_gL}
\end{equation}
\begin{equation}
|g_{L}^{(i)}(t)|\leq(3iL^{\frac{3}{2}})^{i},\label{eq:bound_on_derivative_of_gL}
\end{equation}
\begin{equation}
\Big|\left(\frac{1}{g_{L}}\right)^{(i)}(t)\Big|\leq2\cdot i!\cdot(C\sqrt{i}L^{\frac{3}{2}})^{i}.\label{eq:bound_on_derivative_of_gLinverse}
\end{equation}
\end{lem}

\begin{proof}
Fix $0\leq t\leq\frac{1}{2L^{2}}$. The inequality $g_{L}(t)\leq1$
is clear and only uses $|t|\leq\frac{1}{L}$. For the lower bound, use
that $(1-u)\geq e^{-2u}$ for $0<u<\frac{1}{4}$ to bound 
\[
g_{L}(t)\geq\exp\left(-2\sum_{c=1}^{L}\frac{L}{c}c^{2}t^{2}\right)=\exp(-L^{2}(L+1)t^{2})\geq\frac{1}{2}.
\]
We now prove (\ref{eq:bound_on_derivative_of_gL}). Let $N=\sum_{c=1}^{L}\lfloor\frac{L}{c}\rfloor$
and write $g_{L}(t)=\prod_{k=1}^{N}(1-c_{k}^{2}t^{2})$, where each integer $c \leq L$ appears $\lfloor\frac{L}{c}\rfloor$ times in the sequence $c_1,\dots,c_N$. Each of the
factors $h_{k}(t)=(1-c_{k}t^{2})$ is a degree $2$ polynomial, so
when we differentiate $i$ times $g_{L}$ using the Leibniz rule,
we obtain a big sum of terms, in which some factors are derived twice
(so equal to $h''_{k}(t)=-2c_{k}^2$), some are derived once (so
equal to $h'(t)=-2c_{k}^2t$) and all the other are not derived.
Gathering the terms according to how many factors are derived twice,
we can write 
\[
\frac{g_{L}^{(i)}}{g_{L}}=\sum_{s=0}^{\lfloor\frac{i}{2}\rfloor}\binom{i}{2s}(2s-1)!!\sum_{k_{1},\dots,k_{s},\ell_{1},\dots,\ell_{i-2s}\textrm{ all distinct }}\prod_{\alpha=1}^{s}\frac{h_{k_{\alpha}}''}{h_{k_{\alpha}}}\cdot\prod_{\beta=1}^{i-2s}\frac{h'_{\ell_{\beta}}}{h_{\ell_{\beta}}}.
\]
The term $\binom{i}{2s}(2s-1)!!$ appears as the number of ways to
partition a set of size $i$ (the steps of derivation) into $s$ sets
of size $2$ (the steps when a factor $h_k$ is derived twice) and
$i-2s$ sets of size $1$ (the steps when a factor $h_k$ is derived
once). Forgetting the condition that the $k_{\alpha}$ and
$\ell_{\beta}$ are distinct, and bounding $(1-c_{k}^{2}t^{2})\leq1$
and $(2s-1)!!\leq i^{\frac{i}{2}}$, we can bound the preceding, in
absolute value, as follows:
\[
|g_{L}^{(i)}(t)|\leq i^{\frac{i}{2}}\sum_{s=0}^{\lfloor\frac{i}{2}\rfloor}\binom{i}{2s}t^{i-2s}\big(\sum_{k=1}^{N}2c_{k}^{2}\big)^{i-s}\leq i^{\frac{i}{2}}\sum_{j=0}^{i}\binom{i}{j}t^{i-j}\big(\sum_{k=1}^{N}2c_{k}^{2}\big)^{i-\frac{j}{2}}.
\]
Using 
\[
\sum_{k=1}^{N}2c_{k}^{2}=\sum_{c=1}^{L}2\left\lfloor\frac{L}{c}\right\rfloor c^{2}\leq L^{2}(L+1)\leq2L^{3},
\]
we obtain 
\[
|g_{L}^{(i)}(t)|\leq i^{\frac{i}{2}}\sum_{j=0}^{i}\binom{i}{j}\left(\frac{1}{2L^{2}}\right)^{i-j}\big(2L^{3}\big)^{i-\frac{j}{2}}=i^{\frac{i}{2}}\left(L+\sqrt{2}L^{\frac{3}{2}}\right)^{i}\leq\left(3\sqrt{i}L^{\frac{3}{2}}\right)^{i}.
\]
This proves (\ref{eq:bound_on_derivative_of_gL}).

The last inequality (\ref{eq:bound_on_derivative_of_gLinverse}) is
a formal consequence of the first two. Indeed, by induction we see
that, for any function $g$ we can write the $i$\textsuperscript{th}
derivative of $\frac{1}{g}$ as a product of $(2i-1)!!$ terms, each
of which is of the form 
\[
\frac{\pm g^{(\alpha_{1})}\cdot\dots\cdot g^{(\alpha_{i})}}{g^{i+1}}
\]
where $(\alpha_{1},\dots,\alpha_{i})$ are nonnegative integers that
sum to $i$. For $g=g_{L}$, by (\ref{eq:bound_on_gL}) and (\ref{eq:bound_on_derivative_of_gL}),
each of these terms is bounded above by $\left(3\sqrt{i}L^{\frac{3}{2}}\right)^{i}2^{i+1}$. 
\end{proof}
We extend the map $w\mapsto\varphi_{w}$ by linearity, setting $\varphi_{x}=\sum_{w\in\F_{r}}x(w)\varphi_{w}$
for $x\in\C[\F_{r}]$.
\begin{lem}
\label{lem:bound_on_derivatives_of_phix} If $x\in\C_{\leq q}[\F_{r}]$,
then for any $i$, 
\[
\sup_{t\in\big[0,\frac{1}{2D_{q}^{2}}\big]}\frac{|\varphi_{x}^{(i)}(t)|}{i!}\leq p(i,q)\|x\|_{C^{*}(\F_{r})},
\]

where
\[
p(i,q)=\begin{cases}
(4K+4)\left(\frac{CD_{q}^{4}}{i^{2}}\right)^{i} & \mathrm{\mathrm{if }}\:i\leq D_{q,}\\
(4K+4)\left(C\sqrt{i}D_{q}^{\frac{3}{2}}\right)^{i} & \mathrm{if }\:i>D_{q}.
\end{cases}
\]
\end{lem}

\begin{proof}
Let $P=g_{Kq}\varphi_{x}$. We know that $P$ is a polynomial of degree
$\leq D_{q}$, so to bound its derivatives using the Markov brothers
inequality, we need to derive an a priori bound on the values of $P$.
For $n\geq Kq$, we have $g_{Kq}(\frac{1}{n})\in[0,1]$ and therefore
\[
\left|P\left(\frac{1}{n}\right)\right|\leq\left|\varphi_{x}\left(\frac{1}{n}\right)\right|=n^{-K}\Big|\E[\Tr(\sigma_{n}(x(u_{1},\dots,u_{r})))]\Big|.
\]
For every $(u_{1},\dots,u_{r})\in\cU(n)^{r}$, the map $w\mapsto\sigma_{n}(w(u_{1},\dots,u_{r}))$
is a unitary representation of $\F_{r}$ (hence of $C^{*}(\F_{r})$),
so $\|\sigma_{n}(x(u_{1},\dots,u_{r}))\|\leq\|x\|_{C^{*}(\F_{r})}$
almost surely. Bounding the trace of a matrix by its norm times its
size, we obtain $|\E[\Tr(\sigma_{n}(x(u_{1},\dots,u_{r})))]|\leq\mathrm{dim}(\sigma_{n})\|x\|_{C^{*}(\F_{r})}$,
and we conclude by the dimension bound (\ref{eq:dimension_of_sigma_n})
that 
\[
\left|P\left(\frac{1}{n}\right)\right|\leq(K+1)\|x\|_{C^{*}(\F_{r})}.
\]
We deduce by Lemma~\ref{lemm:Markov_after_epsilon_net} that for
any integer $j$, 
\begin{equation}
\sup_{t\in[0,\frac{1}{2D_{q}^{2}}]}\frac{|P^{(j)}(t)|}{j!}\leq(2K+2)\left(\frac{CD_{q}^{4}}{j^{2}}\right)^{j}\|x\|_{C^{*}(\F_{r})}.\label{eq:bound_on_derivatives_of_P}
\end{equation}
By the Leibniz rule, remembering that $P$ is a polynomial of degre
$\leq D_{q}$,
\[
\frac{\varphi_{x}^{(i)}}{i!}=\frac{1}{i!}\sum_{j=0}^{\min(i,D_{q})}\binom{i}{j}P^{(j)}\left(\frac{1}{g_{Kq}}\right)^{(i-j)}.
\]
We see from Lemma~\ref{lem:bounds_on_g_L} and (\ref{eq:bound_on_derivatives_of_P}) that, on $[0,\frac{1}{2D_{q}^{2}}]$, this is bounded above by
\[(4K+4) C^i \|x\|_{C^*(\F_r)} \sum_{j=0}^{\min(i,D_{q})}\left(\frac{D_q^4}{j^2}\right)^j \big(\sqrt{i-j} (Kq)^{\frac 3 2}\big)^{i-j}.\]
We bound $D_q^4 \leq D_q^{\frac{7}{2}} \sqrt{\max(D_q,i)}$ and $\sqrt{i-j} (Kq)^{\frac 3 2} \leq D_q^{\frac 3 2}  \sqrt{\max(D_q,i)}$, so we obtain
\[ \frac{|\varphi_{x}^{(i)}(t)|}{i!} \leq (4K+4) (C D_q^{\frac 3 2}\sqrt{\max(D_q,i)})^i \|x\|_{C^*(\F_r)} \sum_{j=0}^{\min(i,D_q)} \left(\frac{D_q^2}{j^2}\right)^j.\]
Write $a_j = \left(\frac{D_q^2}{j^2}\right)^j$. For every $j<D_q$, we have $a_j \leq e^2 a_{j+1}$, so $\sum_{j=0}^{\min(i,D_q)} a_j \leq (1+e^2)^i a_{\min(i,D_q)}$ and the lemma follows.
\end{proof}

We deduce two facts from this lemma. In the first, the temperedness of $v_i$ is where most of the results of this paper are used. It is worth noting here that, in the particular case of $K=1$, a much stronger result is known: all the $v_i$'s are tempered, as proven by Parraud \cite{parraud2023asymptotic}. We will see below that they are even completely tempered, see Proposition~\ref{prop:tempered_implies_Ctempered}.%, notably the difficult conclusion \ref{enu:generic-decay} in Theorem~\ref{thm:summary_of_characters}, and the random walk results from Section~\ref{sec:RW}.
\begin{lem}
  \label{lem:derivatives_tempered} For every integer $i\geq0$, let $v_i\colon w \in  \F_r  \mapsto \frac{\varphi_{w}^{(K+i)}(0)}{(K+i)!}$. There is a polynomial $P$ of degree $4K+Ki+1$ such that $|v_i(x)| \leq P(q) \|x\|_{C^*(\F_r)}$ for every $q$ and every $x \in \C_{\leq q}[\F_r]$.
%% the  map $v_{i}\colon x\mapsto\frac{\varphi_{x}^{(K+i)}(0)}{i!}$
%% extends to a bounded linear map on
  %% $\mathcal{S}_{4i+4K+1}(\F_{r})$.
  Moreover, $v_i$ is tempered if $i<\frac K 6$.
\end{lem}
\begin{proof}
The existence of $P$ is a consequence of Lemma~\ref{lem:bound_on_derivatives_of_phix} and the bound $D_{q}\leq CKq\log(1+K)\log(1+q)$, which implies $\sup_{q\geq1}\frac{p(i,q)}{(1+q)^{4i+1}}<\infty$.

Let us explain why $v_i$ is tempered if $i<\frac K 6$. The proof combines several ingredients: the first is Haagerup's inequality, which tells us that we can apply the criterion for temperedness given in Proposition~\ref{prop:criterion_for_tempered_inRDgroups}. The second is the polynomial growth of $v_i$: as a particular case of what we have just shown, there is a constant $C_i$ such that for every $q$ and every $w \in \F_r$ in the ball of radius $q$, $|v_i(w)| \leq C_i (1+q)^{j}$, where $j=4i+4K+1$. The third ingredient is the small support of $v_i$ given by item~\ref{item:derivatives_not_proper_powers} in Theorem~\ref{thm:random-matrix-est}. The last is the random walk results from Section~\ref{sec:RW}. 

Let us put all these ingredients together. Let $\mu$ be a reasonable probability measure on $\F_r$, and $(\gamma_n)_n$ be the associated random walk on $\F_r$. If $\mu$ is supported in the ball of radius $q$, we know that $\gamma_n$ belongs to the ball of radius $qn$, so that
\[ \E |v_i(\gamma_n)| \leq C_i (1+qn)^j \prob( v_i(\gamma_n) \neq 0) \leq C_i C(\mu) (1+qn)^j n^5 \rho(\mu)^n.\]
We deduce
\[ \limsup_n (\E |v_i(\gamma_n)|)^{\frac 1 n} \leq \rho(\mu),\]
so $v_i$ is tempered by Proposition~\ref{prop:criterion_for_tempered_inRDgroups}.
\end{proof}
The second consequence is an analogue in our context to the master
inequality in \cite[Theorem 7.1]{chen2024new}. In the particular case
$K=1$, a similar result was obtained in \cite{parraud2023asymptotic}.
%% We use the notation from
%% Proposition~\ref{prop:Sobolev_group_algebra}, for $0<\varepsilon\leq1$
%% whose value is not important, say $\varepsilon=1$.

\begin{lem}
\label{lem:Taylor_expansion_in_Schwartz_space} 
Let\textup{ ${w}(q)=\exp\left(\frac{q}{\log(2+q)^{2}}\right)$.}
For all integers $r,q,n\geq 1$ and every $x\in\C_{\leq q}[\Gamma]$
\begin{multline}
\Big|\E\Tr\big(\sigma_{n}(x(u_{1},\dots,u_{r}))-\tau(x)\mathrm{Id}\big)-\sum_{i=0}^{r-1}\frac{v_i(x)}{n^{i}}\Big|\\
\leq \frac{(C(K+r)^2 \log(1+K+r)^{12} K^4 \log (1+K)^{4})^{K+r}}{n^r}w(q)\|x\|_{C^{*}(\F_{r})}.\label{eq:asymptotics_of_traces}
\end{multline}
%% Let\textup{ ${w}(q)=\exp(\frac{q}{\log(2+q)^{1+\varepsilon}})$.}
%% For every integers $r,q,n\geq 1$ and every $x\in\C_{\leq q}[\Gamma]$
%% \begin{multline}
%% \Big|\E\Tr\big(\sigma_{n}(x(u_{1},\dots,u_{r}))-\tau(x)\mathrm{Id}\big)-\sum_{i=0}^{r-1}\frac{\varphi_x^{(K+i)}(0)}{(K+i)!n^{i}}\Big|\\
%% \leq \frac{(C(K+r)^2 \log(1+K+r)^{8+4\varepsilon} K^4 \log (1+K)^{4})^{K+r}}{n^r}w(q)\|x\|_{C^{*}(\F_{r})}.\label{eq:asymptotics_of_traces}
%% \end{multline}
\end{lem}

\begin{proof}
If $n\geq Kq$, by (\ref{item:def_of_phiw}) in Theorem~\ref{thm:summary_of_characters},
we know that the left-hand side of (\ref{eq:asymptotics_of_traces})
is equal to 
\[
n^{K}\Big|\varphi_{x}(\frac{1}{n})-\sum_{i=0}^{K+r-1}\frac{\varphi_{x}^{(i)}(0)}{i!n^{i}}\Big|.
\]
By Taylor's inequality and Lemma~\ref{lem:bound_on_derivatives_of_phix},
if we moreover assume that $n\geq2D_{q}^{2}$, this is less than 
\begin{equation}
\frac{1}{n^{r}}p(K+r,q)\|x\|_{C^{*}(\F_{r})}.\label{eq:general_upper_bound_for_trace}
\end{equation}
If $n\leq2D_{q}^{2}$, we claim that this bound is still valid. Indeed, we can bound the left-hand side of (\ref{eq:asymptotics_of_traces}) by 
\[
2\mathrm{dim}(\sigma_{n})\|x\|_{C^{*}(\F_{r})}+\sum_{i=0}^{r-1}\frac{|v_i(x)|}{n^{i}}.
\]
By Lemma~\ref{lem:bound_on_derivatives_of_phix} and \eqref{eq:dimension_of_sigma_n}, this is less than
\[ n^K(p(0,q) + \sum_{i=0}^{r-1} \frac{1}{n^{K+i}}p(K+i,q)) \|x\|_{C^*(\F_r)},\]
which is indeed bounded
above by (\ref{eq:general_upper_bound_for_trace}), at least if the
constant $C$ appearing there is large enough: for example $C\geq 2 e^2$ guarantees  that $\frac{1}{n^{j+1}}p(j+1,q) \geq 2 \frac{1}{n^{j}}p(j,q)$ for every $j$. The statement of the lemma follows because
\[
\sup_{q}p(i,q)\exp\Big(-\frac{q}{\log(2+q)^{2}}\Big)\leq(CK^{4}i^{2}\log(1+K)^{4}\log(1+i)^{12})^{i}.\qedhere
\]
\end{proof}

We can now deduce the main result of this section. Now we make $K$
vary, so we write $\sigma_{K,n}$ the representation that was denoted
$\sigma_{n}$ so far, and $v_{K,i}$ the function denoted $v_i$ so far.

\begin{lem}
\label{lem:convergence_in_probability} Let $A<\frac{1}{42}$. Then
for every $x\in\C[\F_{r}]$, 
\[
\lim_{n}\sup_{K\leq n^{A}}\E\Big|\|\sigma_{n,K}(x(u_{1},\dots,u_{r}))\|-\|x\|_{C_{\lambda}^{*}(\F_{r})}\Big|=0.
\]
\end{lem}

\begin{proof}
  Let $1 \leq K_n \leq n^A$ for every $n$, and let $\pi_n$ be the random representation of $\F_r$: $\pi_n(w) = \sigma_{K,n}(w(u_1,\dots,u_r))$. Apply Lemma~\ref{lem:Taylor_expansion_in_Schwartz_space} for the smallest integer $r \geq \frac K 6$: for every $q$ and every $x \in \C_{\leq q}[\Gamma]$,
\[\Big|\E\Tr (\pi_n(x)) - u_n(x)\Big| \leq \varepsilon_n  w(q)\|x\|_{C^{*}(\F_{r})}.\]
where
\[ \varepsilon_n = \frac{(C K_n^6 (\log K_n)^{12+8\varepsilon})^{K_n+r}}{n^r} \leq (C n^{42A-1} (\log n)^{12+8\varepsilon})^{\frac{1}{6}} \to 0\]
and
\[ u_n(x) = \dim(\sigma_{n,k})\tau(x) + \sum_{i=0}^{r-1} \frac{v_{K_n,i}(x)}{n^i}.\]
By Lemma~\ref{lem:derivatives_tempered}, the function $u_n$ is tempered as a finite sum of tempered function, and it satisfies
\[ |u_n(x)| \leq C(n) (1+q)^{4r+4K_n-3}.\]

Moreover, $C^*_\lambda(\Gamma)$ has a unique trace (it is even simple \cite{powers}). So all the hypotheses of Proposition~\ref{prop:criterion_for_strong_convergence} are satisfied; its conclusion proves the lemma. 
\end{proof}

\subsection{From convergence in probability to almost sure convergence}

Lemma~\ref{lem:convergence_in_probability} is not exactly our main
result, because we have convergence in expected value.
However, by the concentration of measure phenomenon in the groups
$\U(n)$ (see \cite[Theorem 5.17]{MR3971582}), we can improve our
results to almost sure convergence. 
\begin{prop}
\label{prop:concentration} For every $x\in\C_{\leq q}[\F_{r}]$,
there is a constant $C(x)=\sum_{w}|x(w)||w|$ such that 
\begin{align*}
 & \prob\Big(\|\sigma_{n,K}(x(U_{1}^{(n)},\dots,U_{r}^{(n)}))\|\geq\E\|\sigma_{n,K}(x(U_{1}^{(n)},\dots,U_{r}^{(n)}))\|+\varepsilon)\Big)\\
 & \leq\exp\left(-\frac{(n-2)\varepsilon^{2}}{24C(x)^{2}K^{2}}\right).
\end{align*}
\end{prop}

\begin{proof}
Equip $\U(n)^{r}$ with the $L_{2}$-sum of Hilbert-Schmidt distances.
The function $(U_{1}^{(n)},\dots,U_{r}^{(n)})\mapsto\|\sigma_{n,K}(x(U_{1}^{(n)},\dots,U_{r}^{(n)}))\|$
is $KC(x)$-Lipschitz, so the bound is \cite[Theorem 5.17]{MR3971582}
(for the original see \cite[Proof of Thm. 3.9]{Voiculescu1991}).
\end{proof}
\begin{proof}[Proof of Theorem \ref{thm:main}]
Theorem \ref{thm:main} now follows by combining Proposition \ref{prop:concentration}
and Lemma \ref{lem:convergence_in_probability}. 
\end{proof}
\subsection{Variants and operator coefficients}
If we only consider the representation $\pi_{k,\ell}$ from the introduction with $\ell=0$, we have the stronger conclusion \ref{enu:poly} in Theorem~\ref{thm:random-matrix-est} instead of \ref{enu:generic-decay}. Taking this improvement into account, with the notation of Theorem~\ref{thm:main}, the conclusion becomes that if $A < \frac 1 {12}$,
\[
\sup_{k\leq n^{A}}\left|\left\Vert \pi_{k,0}\left(p\left(U_{1}^{(n)},\ldots,U_{r}^{(n)}\right)\right)\right\Vert -\|p(x_{1},\ldots,x_{r})\|\right|=o(1).
\]
This improvement illustrates that the more of the higher derivatives $v_{K,i}$ from Lemma~\ref{lem:derivatives_tempered} are shown to be tempered, the stronger the conclusion. For example, if we knew that $v_{K,i}$ is tempered for every $i$ and every $K$ (which Parraud has proved in the case $K=1$ \cite{parraud2023asymptotic}), or at least for every $i \leq K/o(1)$, then the condition on $A$ in Theorem~\ref{thm:main} would be $A<\frac 1 6$.

\begin{thm}\label{thm:main_operator_coefficients} Let $A<\frac 1 6$ and for every $n$, let $K_n \leq n^{A}$. Assume that $v_{K_n,i}$ is tempered for every $n$ and every $i$. Let $k_n = \exp(n^{\frac 1 2 - 2A}(\log n)^{-4})$. For every $q$ and every sequence $y_n \in M_{k_n}\otimes \C_{\leq q}(\F_r)$, almost surely
  \[ \lim_n \frac{\|(\mathrm{id}\otimes \sigma_{K_n,n})(y_n)\|}{\|(\mathrm{id}\otimes \lambda)(y_n)\|}=1.\]
\end{thm}
If the assumption was that $v_{K_n,i}$ is completely tempered, the proof of the theorem would be a straightforward adaptation of the proof of Theorem~\ref{thm:main}. So our main task is to prove the following:
\begin{prop}\label{prop:tempered_implies_Ctempered} For every $K,I$, if $v_{K,0},\dots,v_{K,I}$ are all tempered, then they are all completely tempered.
\end{prop}
For the proof, we need the following result. Here if $\mathcal{A}$ is an algebra, $\mathcal{A}^n$ denotes its subalgebra equal to the linear span of $\{x_1 \dots x_n \mid x_i \in \mathcal{A}\}$.
\begin{lem}\label{lem:positive_forms_on_staralgebras} Let $\mathcal{A}$ be a $*$-algebra and $u \colon \mathcal{A} \to \C$ a positive linear map: $u(x^*x)\geq 0$ for every $x \in \mathcal{A}$. If there is $i$ such that $u(x^i)=0$ for every self-adjoint $x$, then $u= 0$ on $\mathcal{A}^3$.
\end{lem}
\begin{proof} The form $\langle a,b\rangle = u(a^*b)$ is a scalar product. By the Cauchy-Schwarz inequality, if $k \geq 2$, if $a^* = x^*xx^* \dots$ ($k+1$) terms and $b= \dots x$ ($k-1$) terms,
  \[ |u((x^*x)^k)|^2 = |\langle a,b\rangle|^2\leq u((x^*x)^{k+1}) u((x^*x)^{k-1}),\]
  so by induction we have $u((x^*x)^2)=0$, and by the Cauchy-Schwarz inequality again, $u(x^* x z)=0$ for every $z \in \mathcal A$. By polarization we deduce $u(xyz)=0$ for every $x,y,z$.
\end{proof}
\begin{proof}[Proof of Proposition~\ref{prop:tempered_implies_Ctempered}]
  We use the notation and results from Section~\ref{subsec:Sobolev}. Lemma~\ref{lem:derivatives_tempered} tells us that $v_{K,i}$ extends by continuity to a linear map $\mathcal{S}_{4K+4i+1}(\F_r) \to \C$. Let $\mathcal{A}_i = \mathcal{S}_{4K+4i+1}(\F_r) \cap \ker\lambda$. We prove by induction that $v_{K,i}$ vanishes on $(\mathcal{A}_i)^{3^{i+1}}$ for every $i\leq I$. This readily implies that $\mathrm{\Tr}\otimes v_{K,i}$ vanishes on $M_k \otimes (\mathcal{A}_i)^{3^{i+1}}$ for all $k$. In particular, we see that for every self-adjoint $x \in M_k\otimes \mathcal{A}_i$, $(\Tr\otimes v_{K,i})(x^{3^{i+1}})=0$ because $x^{3^{i+1}}$ belongs to $M_k \otimes (\mathcal{A}_i)^{3^{i+1}}$.  Proposition~\ref{prop:characterization_of_temperedness_in_Sobolev_algebras} therefore implies that $v_{K,i}$ is completely tempered.

So let us assume that $v_{K,j}$ vanishes on $(\mathcal{A}_j)^{3^{j+1}}$ for every $j<i$ (we assume nothing if $i=0$), and let us show that it is true for $j=i$.
 
Using the induction hypothesis, it follows from Lemma~\ref{lem:Taylor_expansion_in_Schwartz_space} that on the $*$-algebra $(\mathcal{A}_i)^{3^i}$ (which is contained in $\ker \tau$ and $(\mathcal{A}_j)^{3^{j+1}}$ for every $j<i$),
  \[ v_{K,i}(x) = \lim_n n^{i} \E \Tr \sigma_{K,n}(x).\]
  In particular, $v_{K,i}$ is positive on  $(\mathcal{A}_i)^{3^i}$ as a limit of positive maps. Moreover, the fact that $v_{K,i}$ is tempered implies that $v_{K,i}(x^n)=0$ for all $n\geq 4K+4i+3$, see Proposition~\ref{prop:characterization_of_temperedness_in_Sobolev_algebras}. By Lemma~\ref{lem:positive_forms_on_staralgebras}, we obtain that $v_{K,i}$ vanishes on $(\mathcal{A}_i)^{3^{i+1}}$. This concludes the proof of the proposition.
  \end{proof}
\begin{proof}[Proof of Theorem~\ref{thm:main_operator_coefficients}]
Thanks to Proposition~\ref{prop:tempered_implies_Ctempered}, the temperedness assumption is automatically upgraded to complete temperedness. We apply Lemma~\ref{lem:Taylor_expansion_in_Schwartz_space} with $r = n^{\frac 1 2 - 2A} (\log n)^{-\frac 7 2}$, which is chosen so that
  \[ k_n \frac{(C(K+r)^2 \log(1+K+r)^{12} K^4 \log (1+K)^{4})^{K+r}}{n^r} \to 0.\]
So by Proposition~\ref{prop:criterion_for_strong_convergence_coefficients} we deduce the convergence in probability. The almost sure convergence is obtained by the concentration of measure phenomenon from Proposition~\ref{prop:concentration}.\end{proof}
For $K_n=1$ (that is, $A=0$), Parraud proved that $v_{1,i}$ are tempered for every $i$. As a consequence, we obtain the unconditional result from Theorem~\ref{thm:large_coefficients_standard}.
%% \begin{thm}\label{thm:large_coefficients_standard} Let $k_n=\exp(n^{\frac 1 2} (\log n)^{-4})$. For every $q$ and every sequence $P_n$ of non-commutative $*$-polynomial of degree $\leq q$ and with coefficients in $M_{k_n}(\C)$, almost surely
%%   \[ \lim_n \frac{\|P_n(U_1^{(n)},\dots,U_r^{(n)})\|}{\|P_n(x_1,\dots,x_r)\|}=1.\]
%%   \end{thm}
%% Bordenave and Collins \cite[Theorem 1.1]{BordenaveCollins3} proved the same theorem for $q=1$ and $k_n=\exp(n^{\frac{1}{32r+160}})$ instead of $k_n=\exp(n^{\frac 1 2 + o(1)})$.

\section{Proof of Corollary \ref{cor:mainSUn} }

Corollary \ref{cor:mainSUn} follows easily from the proof of Theorem
\ref{thm:main} and the following lemma:
\begin{lem}
\label{lem:comparisonUnSUn}For every partitions $\lambda\vdash k,\mu\vdash\ell$,
every $w\in\F_{r}$ and every $n>(k+\ell)|w|$, we have
\end{lem}

\[
\E_{n}[s_{\lambda,\mu}(w)]=\int_{\SU(n)^{r}}s_{\lambda,\mu}(w(v_{1},v_{2},\ldots,v_{r}))dv_{1}\cdots dv_{r}.
\]

\begin{proof}
Let $z_{1},\dots,z_{r},z'_{1},\dots,z'_{r},v_{1},\dots,v_{r}$ be
independent random variables where $z_{i}$ are uniform in the center
of $\U(n)$ (the complex numbers of modulus one), $z'_{i}$ are uniform
in the center of $\SU(n)$ (the $n$-th roots of unity) and $v_{i}\in\SU(n)$
are Haar distributed, so that $(z_{1}v_{1},\dots,z_{r}v_{r})$ are
independent Haar-distributed variables in $\U(n)$, and $(z'_{1}v_{1},\dots,z'_{r}v_{r})$
are independent Haar-distributed variables in $\SU(n)$. 

By Schur's lemma, for every $z$ in the center of $U(n)$, $\pi_{\lambda,\mu}(z)$
is a multiple of the identity. On the other hand, we know that $(\lambda_{1},\dots,\lambda_{p},\dots,-\mu_{q},\dots,-\mu_{1})$
is the maximal weight of $\pi_{\lambda,\mu},$ so the corresponding
character of the maximal torus appears in the restriction of $\pi_{\lambda,\mu}$.
This means that the scalar $\pi_{\lambda,\mu}(z)$ is $z^{\lambda_{1}+\dots+\lambda_{p}-\mu_{q}-\dots-\mu_{1}}=z^{k-\ell}$
for every $z$ in the center of $\U(n)$. Putting these two facts
together, we see that
\begin{align*}
\E_{n}[s_{\lambda,\mu}(w)] & =\E[s_{\lambda,\mu}(w(z_{1}v_{1},\dots,z_{r}v_{r}))]
\end{align*}
decomposes by independence as the product 
\begin{align*}
\E_{n}[s_{\lambda,\mu}(w)]=\E[s_{\lambda,\mu}(w(v_{1},\dots,v_{r}))]\cdot\E[w(z_{1,}\dots,z_{r})^{k-\ell}],
\end{align*}
and similarly
\begin{align*}
\E[s_{\lambda,\mu}(w(z'_{1}v_{1},\dots,z'_{r}v_{r}))]=\E[s_{\lambda,\mu}(w(v_{1},\dots,v_{r}))]\cdot\E[w(z'_{1},\dots,z'_{r})^{k-\ell}].
\end{align*}
But $\E w(z_{1,}\dots,z_{r})^{k-\ell}$ is equal to $1$ if $w^{k-\ell}$
vanishes in the abelianization $\Z^{r}$ of $\F_{r}$, and is equal
to $0$ otherwise. Similarly, $\E w(z'_{1,}\dots,z'_{r})^{k-\ell}$
is equal to $1$ if $w^{k-\ell}$ vanishes in $(\Z/n\Z)^{r}$ and $0$
otherwise. These two conditions coincide if $n>|k-\ell| \cdot |w|$, and
in particular if $n>(k+\ell)|w|$.
\end{proof}
\begin{proof}[Proof of Corollary \ref{cor:mainSUn}]
 Using Lemma \ref{lem:comparisonUnSUn}, we see that Lemma \ref{lem:Taylor_expansion_in_Schwartz_space}
also holds if $\E_{n}$ denotes the integration over $\SU(n)^{r}$:
the validity of (\ref{eq:general_upper_bound_for_trace}) when $x\in\C_{\leq q}[\F_{r}]$
with $n\geq2D_{q}^{2}$ is by Lemma \ref{lem:comparisonUnSUn}, the
extension for arbitrary $x$ follows with the same proof. Therefore,
Lemma \ref{lem:convergence_in_probability} also holds for $\SU(n)$.
Moreover, the concentration of measure from Proposition \ref{prop:concentration}
also holds for $\SU(n)$ (same reference as there), so Theorem \ref{thm:main}
holds also for independent variables in $\SU(n).$ Corollary \ref{cor:mainSUn}
follows by Corollary \ref{cor:dimensions_SUn_representations}. 
\end{proof}

\section{Transverse maps\label{sec:Transverse-maps}}

We now fix the free group we work with $\F_{r}$ and its generators
$\{x_{1},\ldots,x_{r}\}$. Here we introduce a framework, following
\cite{MageePuder1}, that we use to prove Theorem \ref{thm:random-matrix-est}.

In the sequel, a marked rose refers to a finite \emph{CW}-complex
structure $R_{r}$ whose underlying topological space is the wedge
of $r$ circles, with the base-point being the wedge point, denoted
$o$, and with a marking $\pi_{1}(R_{r},o)\cong\F_{r}$ that identifies
the generators $\{x_{1},\ldots,x_{r}\}$ with oriented circles of
the rose.

The following definition will not be used in the paper but gives important
background motivation for the definitions to follow.
\begin{defn}[Motivational]
A transverse map is a manifold $M$, a based marked rose $R_{r}$
and a continuous function
\[
f:M\to R_{r}
\]
transverse to all the vertices of $R_{r}$. 
\end{defn}

In the rest of the paper we work with \emph{classes of transverse
maps. }Instead of getting into what is an isotopy of transverse maps,
etc., we make the following equivalent combinatorial definition.
\begin{defn}
\label{rem:ribbon-graph}An \emph{(isotopy) class} of (filling) transverse
map from a surface with boundary to a marked rose $R_{r}$ is the
finitary combinatorial data:
\begin{itemize}
\item a marked base-point and orientation for each boundary component
\item a ribbon graph structure on the surface, where vertices are discs
and edges correspond to interfaces between discs on their boundaries;
such an interface is called an \emph{arc}
\item information about which arc is in the preimage of each vertex of $R_{r}$
other than $o$, and which sides of the vertex correspond to which
sides of the arc
\item up to isomorphism of the above data under decoration respecting homemorphisms.
\end{itemize}
We refer to these simply as \emph{classes of transverse maps.}
\end{defn}

\begin{defn}[Strict transverse maps]
A class of transverse map $f$ as in Definition \ref{rem:ribbon-graph}
is \emph{strict} if for every two vertices $p_{1},p_{2}$ of $R_{r}$
that are consecutive on some circle (and not $o$), do \uline{not}
have parallel preimages in the following sense. Suppose the orientation
of their circle in $R_{r}$ points from $p_{1}$ to $p_{2}$. The
preimages of $p_{1}$ and $p_{2}$ are \emph{parallel} if for every
arc $\alpha$ in the preimage of $p_{1}$, there is a parallel\footnote{Cobounding a rectangle.}
arc in the preimage of $p_{2}$ on the side of $\alpha$ `towards'
$p_{2}$. 
\end{defn}

Given a class $[f]$ of transverse map on surface with boundary there
is an obvious way to get a boundary class of transverse map on a union
of circles, denoted $\partial[f]$. 

Given $\kappa:[r]\to\N\cup\{0\}$ let $R_{r}^{\kappa}$ be the (isomorphism
class of) rose with $\kappa(i)+1$ vertices in the interior of the
circle corresponding to $x_{i}$. Let $|\kappa|=\sum_{i\in[r]}\kappa_{i}$. 

Given every non-identity $w\in F_{r}$, let $\mathfrak{\mathfrak{w}}:S^{1}\to R_{r}$
denote a fixed immersion\footnote{Away from the base point of $S^{1}$, where there could be backtracking
if $w$ is not cyclically reduced.} from a based oriented circle to the based rose such that if $\gamma$
is the generator of $\pi_{1}(S^{1},\mathrm{basepoint})$ corresponding
to the chosen orientation, 
\[
\mathfrak{w}_{*}(\gamma)=w\in\pi_{1}(R_{r},o).
\]

\section{Proof of Theorem \ref{thm:random-matrix-est} Part \ref{enu:deg-of-poly}}

Suppose $|w|\leq q$, $\lambda\vdash k,\mu\vdash\ell$ as before.
By a result of Koike \cite[eq. (0.3)]{Koike} it is possible to write
for $g\in\U(n)$
\[
s_{\lambda,\mu}(g)=\sum_{|\lambda'|\leq k,|\mu'|\leq\ell}\alpha_{\lambda',\mu'}^{\lambda,\mu}s_{\lambda'}(g)s_{\mu'}(g^{-1})
\]
where the coefficients $\alpha_{\lambda'.\mu'}^{\lambda,\mu}$ are
given in \emph{(ibid.)} as sums of products of Littlewood--Richardson
coefficients and importantly, do not depend on $n$. Then by base
change between Schur polynomials as above and power sum symmetric
polynomials, we obtain
\begin{equation}
s_{\lambda,\mu}(w)=\sum_{|\lambda'|\leq k,|\mu'|\leq\ell}\beta_{\lambda',\mu'}^{\lambda,\mu}p_{\lambda'}(w)p_{\mu'}(w^{-1})\label{eq:power-sum-expansion}
\end{equation}
 where again the coefficients do not depend on $n$. This tells us
the poles of $\E[s_{\lambda,\mu}(w)]$ are at most the union of every
possible pole of 
\[
\E_{n}[p_{\lambda'}(w)p_{\mu'}(w^{-1})]
\]
 with $k'=|\lambda'|\leq k,\ell'=|\mu'|\leq\ell.$ 

For the purpose of bounding the poles, we can here use the Weingarten
calculus in the most naive possible way (later, we need to do something
more sophisticated). To this end, \cite[Theorem 2.8]{MageePuder1}
yields
\begin{equation}
\E_{n}[p_{\lambda'}(w)p_{\mu'}(w^{-1})]=\sum_{\substack{[f:\Sigma\to R_{r}^{(2)}]\\
\partial f\cong\mathfrak{w}\circ\varphi_{\lambda',\mu'}
}
}\left(\prod_{1\leq i\leq r}\mathrm{Wg}_{(k'+\ell')L_{i}(w)}\left(\pi_{f,i}\right)\right)n^{N(f)}\label{eq:sum_of_wg_function}
\end{equation}
where $R_{r}^{(2)}=R_{r}^{(\kappa)}$ with $\kappa=(2,2,\ldots,2)$,
$\pi_{f,i}\in S_{(k'+\ell')L_{i}(w)}$ are permutations determined
from the combinatorial structure of $f$, and $N(f)\in\N$ is similar
--- neither depend on $n$. It is a finite sum.

\uline{}

Each term
\[
\mathrm{Wg}_{(k'+\ell')L_{i}(w)}\left(\pi_{f,i}\right)\in\Q(n)
\]
is given by the formula for the \emph{Weingarten function }\cite[eq. (9)]{CollinsSniady}
\begin{equation}
\mathrm{Wg}_{L}(\pi)=\frac{1}{(L!)^{2}}\sum_{\lambda\vdash L}\frac{\chi_{\lambda}(1)^{2}}{s_{\lambda}(1)}\chi_{\lambda}(\pi);\label{eq:Wg-def}
\end{equation}
here we are viewing $s_{\lambda}(1)$ as the formal element of $\Q(n)$
\[
s_{\lambda}(1)\eqdf\frac{\prod_{\square\in\lambda}(n+c(\square))}{\prod_{\square\in\lambda}h_{\lambda}(\square)},
\]
$c(\square)$ is the content of the box,
\global\long\def\F{\mathbf{F}}%
\[
c(\square)=j(\square)-i(\square)
\]
and $h_{\lambda}(\square)$ is the hook-length of the box. The only
terms in the formula for the Weingarten function that depend on $n$
are therefore the denominators $\prod_{\square\in\lambda}(n+c(\square))$.

If $c\geq0$ then the factor $(n+c)$ appears at most $d(c,L)$
times in the denominator of $\mathrm{Wg}_{L}(\pi)$ where $d(c,L)$
is the largest natural number $d$ such that 
\[
d(d+c)\leq L.
\]
Clearly $d(c,L)\leq\frac{L}{c}$ for $c>0$. If $c<0$ then a similar
argument says $(n+c)$ appears at most $d(n,-c)$ times.

This means that $\prod_{\square\in\lambda}(n+c(\square))$ can be
written as $n^{L}Q_{\lambda}(\frac{1}{n})$ where $Q_{\lambda}$ is
a polynomial that divides $g_{L}(\frac{1}{n})$ where (as in (\ref{eq:g-def}))
\[
g_{L}(x)\eqdf\prod_{c=1}^{L}(1-c^{2}x^{2})^{\lfloor\frac{L}{c}\rfloor}.
\]
Therefore, $\frac{g_{L}(\frac{1}{n})}{\prod_{\square\in\lambda}(n+c(\square))}=\frac{1}{n^{L}}\frac{g_{L}}{Q_{\lambda}}(\frac{1}{n})$
is a polynomial in $\frac{1}{n}$ of degree $\leq L + \operatorname{deg}(g_L)$.
By \eqref{eq:Wg-def} the same is true for $g_{L}(\frac{1}{n})\mathrm{Wg}_{L}(\pi)$.
Since 
\[
\prod_{i=1}^{r}g_{(k'+\ell')L_{i}}
\]
divides $g_{(k+\ell)q},$ we see from (\ref{eq:power-sum-expansion})
that for $n\geq(k+\ell)q$, the rational function agreeing with $g_{(k+\ell)q}(\frac{1}{n})\E[s_{\lambda,\mu}(w)]$
in this range can be written as a sum of a polynomial in $\frac{1}{n}$
of degree $\leq(k+\ell)q+\operatorname{deg}(g_{(k+\ell)q})$ and possibly
a polynomial in $n$ (because of the terms $n^{N(f)}$). However,
the polynomial in $n$ is necessarily constant because it is known
(see Theorem 1.7 in \cite[Thm. 1.7]{MageePuder1}) that $g_{(k+\ell)q}\left(\frac{1}{n}\right)\E[s_{\lambda,\mu}(w)]$
remains bounded as $n\to\infty$.

To conclude, observe that we can bound the degree of $g_{L}$ by
\[
2\sum_{c=1}^{L}\lfloor\frac{L}{c}\rfloor\leq2\sum_{c=1}^{L}\frac{L}{c}\leq2L(1+\log L),
\]
so that $(k+\ell)q+\operatorname{deg}(g_{(k+\ell)q})\leq3L(1+\log L)$.

\section{Proof of Theorem \ref{thm:random-matrix-est} Part \ref{enu:generic-decay}
and \ref{enu:poly}\label{sec:Proof-of-Theorem-random-matric}}

In the paper \cite{MageeRURSGII} the integrals 
\[
\int_{U(n)^{4}}\Tr(\tilde{w}(u_{1},u_{2},u_{3},u_{4}))s_{\lambda,\mu}([u_{1},u_{2}][u_{3},u_{4}])du_{1}du_{2}du_{3}du_{4}
\]
w.r.t. Haar measure are calculated in a very specific way that exhibits
cancellations that cannot be easily seen\footnote{In fact, we do not know how to.}
using the formula (\ref{eq:power-sum-expansion}). The initial part
of this calculation is very general and does not depend on $\tilde{w}$
or the word $w=[x_{1},x_{2}][x_{3},x_{4}]$. So the method generalizes
as-is to the case when the $\Tr(\tilde{w})$ term is not present and
$[x_{1},x_{2}][x_{3},x_{4}]$ is replaced by any reduced word $w$.
That is, to integrals of the form
\[
\E_{n}[s_{\lambda,\mu}(w)]\eqdf\int_{U(n)^{r}}s_{\lambda,\mu}(w(u_{1},\ldots,u_{r}))du_{1}\cdots du_{r}
\]
 where $w$ is any element of $\F_{r}$, viewed as a reduced word
in the generators. 

The analog of \cite[Cor. 4.7]{MageeRURSGII} here is the following
bound. Let $R_{r}^{0}$ be the based $CW$-complex structure on $R_{r}$
with one interior vertex per circle (and a vertex for the wedge point
$o$). Recall the immersion $\mathfrak{w}:(S^{1},\mathrm{basepoint})\to(R_{r}^{0},o)$
from $\S$\ref{sec:Transverse-maps}. Let 
\[
\varphi_{\lambda,\mu}:\coprod_{i=1}^{\ell(\lambda)}S^{1}\sqcup\coprod_{i=1}^{\ell(\mu)}S^{1}\to S^{1}
\]
denote the map whose components are 
\begin{itemize}
\item multiplication by $\lambda_{i}$ on the $i$\textsuperscript{th}
circle, if $i\leq\ell(\lambda)$, 
\item and multiplication by $-\mu_{i}$ (so orientation reversing) on the
$(i-\ell(\lambda))$\textsuperscript{th}circle for $i>\ell(\lambda)$. 
\end{itemize}
This map is a convenient way of book-keeping winding numbers.
\begin{prop}
\label{prop:Magee-decay-rate}We have 
\[
\E_{n}[s_{\lambda,\mu}(w)]=O_{k,\ell,w}(n^{\max_{w,k,l}\chi(\Sigma)})
\]
where $\max_{w,k,\ell}$ is the maximum over classes of strict transverse
maps $[f]$ from a surface with boundary to $R_{r}^{0}$ such that
\begin{description}
\item [{~}] $\Sigma$ is an oriented surface with compatibly oriented
boundary
\item [{~}] For some $\lambda'\vdash k$ and $\mu'\vdash\ell$, $\partial[f]$
factors as 
\begin{equation}
\partial[f]=[\mathfrak{w}\circ\varphi_{\lambda',\mu'}]\label{eq:boundary-map-form}
\end{equation}
where $[\mathfrak{w}\circ\varphi_{\lambda',\mu'}]$ is the (unique)
class of transverse map to $R_{r}^{0}$ obviously associated to the
immersion $\mathfrak{w}\circ\varphi_{\lambda',\mu'}$. Components
of $\partial\Sigma$ mapped in an orientation respecting (resp. non-orientation
respecting) way by $\varphi_{\lambda',\mu'}$ are called positive
(resp. negative).
\item [{Forbidden~Matchings}] In the factorization of $\partial [f]$ above,
no such arc has endpoints in a positive and negative component of
$\partial\Sigma$ that are mapped to the same point under $\varphi_{\lambda',\mu'}$.
\end{description}
\end{prop}

This proposition follows formally from \cite[\S\S 3.1 - \S\S 4.3]{MageeRURSGII},
but we give the details in the Appendix, $\S$\ref{sec:Appendix:-Extension-of}.
The \textbf{Forbidden Matchings} property is completely crucial and
the main point of the method introduced in \emph{(ibid.). }The following
topological result is a new input to that method.
\begin{prop}
\label{prop:key-top-prop}If $w$ is not a proper power in $\F_{r}$
and cyclically reduced, then any class of transverse map $[f]$ on
underlying surface $\Sigma$ satisfying the conditions of Proposition
\ref{prop:Magee-decay-rate} has 
\[
\chi(\Sigma)\leq-\frac{k+\ell}{6}.
\]
\end{prop}

\begin{cor}[Theorem \ref{thm:random-matrix-est} Part \ref{enu:generic-decay}]
\label{cor:part3}If $w$ is not a proper power in $\F_{r}$ then
\[
\E_{n}[s_{\lambda,\mu}(w)]=O_{k,\ell,w}\left(n^{-\frac{1}{6}\left(k+\ell\right)}\right).
\]
\end{cor}

N.B. One can assume without loss of generality in Corollary \ref{cor:part3}
that $w$ is cyclically reduced, as conjugating $w$ does not change
$\E_{n}[s_{\lambda,\mu}(w)]$.

\subsection{Proof of Proposition \ref{prop:key-top-prop}\label{subsec:Proof-of-Proposition}}

Suppose $w$ is cyclically reduced and not a proper power or the identity,
with word length $|w|$. Suppose that some class of transverse map
$[f:\Sigma\to R_{r}^{0}]$ satisfies all the properties of Proposition
\ref{prop:Magee-decay-rate}.

The transverse map $f$ is a (isomorphism class of)
ribbon graph $\mathcal{R}$ whose edges correspond to the arcs of
$f$. Each edge of the ribbon graph inherits a direction and labeling
by $\{x_{1},\ldots,x_{r}\}$ from the manner in which $f$ crosses
the (sole) interior vertex of the $i$\textsuperscript{th} circle
in $R_{r}^{0}$ along this edge. 

Say that a vertex $v$ of $\mathcal{R}$ is a \emph{topological vertex
}if it has valence $d(v)\geq3$. A \emph{topological edge }$e$ is
a maximal chain of edges incident at vertices of valence $2$. (As
$w$ is cyclically reduced, there are no vertices of valence $1$).

Let $V^{*}$, $E^{*}$ denote the topological vertices and edges respectively.
We have
\[
2E^{*}=\sum_{v\in V^{*}}d(v)\geq3V^{*}
\]
and hence
\begin{equation}
-\chi(\Sigma)=E^{*}-V^{*}\geq\frac{E^{*}}{3}.\label{eq:chi-E-bound}
\end{equation}
We now aim to show $E^{*}$ is large under the previous assumptions.

Every topological edge $e$ is bordered by two segments of $\partial\mathcal{R}$
that each spell a subword of $w^{k_{i}}$ or $w^{-\ell_{i}}$ for
some $i$, reading along the fixed orientation of the boundary. 

Let $L(e)$ denote the number of edges in the chain defining $e$. Suppose that for some topological edge $e$, $L(e)>|w|$. For a reduced
word $v$ in the generators of $\F_{r}$, write $\overline{v}$ for
its mirror (reversing order and replacing generators by inverses).
Then both boundary segments spell words
\[
u_{1}w^{a}u_{2},\quad v_{1}w^{b}v_{2}
\]
 with $a,b\in\Z\backslash\{0\}$, $|u_{1}|,|u_{2}|,|v_{1}|,|v_{2}|<|w|$. 

\textbf{Case 1. If $a,b$ have same sign. }In this case one arrives
at
\[
wu=v\overline{w}
\]
with $|u|,|v|<w$. (see Figure \ref{fig:Auxiliary-to-proof}).
\begin{figure}
\centering{}\includegraphics[scale=1.3]{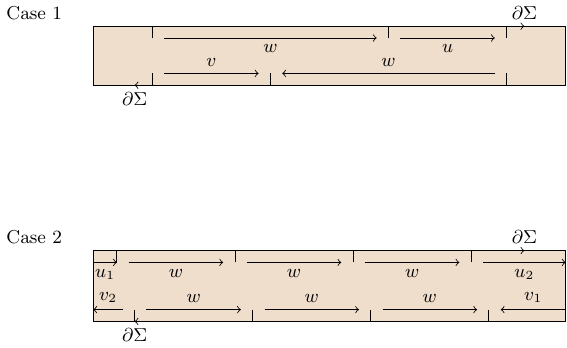}\caption{\label{fig:Auxiliary-to-proof}Auxiliary to proof of Proposition \ref{prop:key-top-prop}.}
\end{figure}
 This is impossible since then $v$ is a prefix of $w$ and $\overline{v}$
is too, meaning $v=\bar{v}$ since they have the same length, so $v$
is empty. Similarly $u$ is empty and $w=\overline{w}$ contradicting
$w$ being the identity. 

\textbf{Case 2. If $a,b$ have opposite signs. }So without loss of
generality assume $a>0$ and $b=-B$ with $B>0$, then we have 
\[
u_{1}w^{a}u_{2}=\overline{v_{2}}w^{B}\overline{v_{1}}.
\]
The \textbf{Forbidden Matchings} property implies that $|u_{1}|\neq|\overline{v_{2}}|$.
Again without loss of generality assume $|u_{1}|<|\overline{v_{2}}|$.
Write $\overline{v_{2}}=u_{1}w_{1}$.

If $|u_{2}|>|\overline{v_{1}}|$ then write $u_{2}=w_{2}\text{\ensuremath{\overline{v_{1}}}}$
to obtain

\begin{equation}
w^{a}w_{2}=w_{1}w^{B}\label{eq:commute}
\end{equation}
as reduced words. This immediately implies $w_{2}=w_{1}$, as both
are prefixes of $w$ of the same length. Since $w$ is not a proper
power, e.g. \cite[Lemma 2.2]{razborov} implies that $w_{1}$ is a
power of $w$, which must be empty since its length is less than $w$.

If $|u_{2}|<|\overline{v_{1}}$\textbar{} then write $w_{2}u_{2}=\overline{v_{1}}$
to arrive at 
\[
w^{a}=w_{1}w^{B}w_{2}.
\]
 This implies $w_{1}$ and $w_{2}$ are both prefixes and suffixes
of $w$. 

Write $w=w_{1}w_{3}$ and obtain
\[
w_{3}w^{a-1}=w^{B}w_{2}.
\]
Since $B>0$, $a>1$ here and we argue as we did from (\ref{eq:commute})
to get a contradiction.

The upshot of this argument is that each topological edge $e$ has
\[
L(e)\le|w|.
\]
Since the total lengths of topological edges is $\frac{1}{2}(k+\ell)|w|$ this
gives $|E^{*}|\geq \frac{k+\ell}{2}$. Hence from (\ref{eq:chi-E-bound})
\[
-\chi(\Sigma)\geq\frac{k+\ell}{6}.
\]

\subsection{Proof of Theorem \ref{thm:random-matrix-est} Part \ref{enu:poly}\label{sec:Improved-bounds-for-poly-char}}

If $\mu=\emptyset$, i.e. $s_{\lambda,\mu}=s_{\lambda}$ is a polynomial
stable character, then instead we can use the following two results.

For $w\in\F_{r}$ the commutator length of $w$, denoted by $\mathrm{cl}(w)\in\N\cup\{\infty\}$
is defined by 
\[
\mathrm{cl}(w)=\inf\{\,g:\,w=[u_{1},v_{1}]\cdots[u_{g},v_{g}]\,,\,u_{i},v_{i}\in\F_{r}\,\}
\]
 interpreted as $\infty$ if it is not possible to write $w$ as a
product of commutators, i.e. $w\notin[\F_{r},\F_{r}]$. The \emph{stable
commutator length} of $w$, denoted $\mathrm{scl}(w)$, is defined
as
\[
\mathrm{scl}(w)\eqdf\lim_{m\to\infty}\frac{\mathrm{cl}(w^{m})}{m}.
\]
On $\F_{r}$, scl takes values in $\Q$ by a result of Calegari \cite{Calegari}.
Duncan and Howie proved in \cite{DuncanHowie} that 
\begin{equation}
\scl(w)\geq\frac{1}{2}\label{eq:duncan-howie}
\end{equation}
 for all $w\in\F_{r}$. This bound can be combined with the following
result of the first named author and Puder following immediately from
\cite[Cor. 1.11]{MageePuder1}.
\begin{thm}
We have 
\begin{equation}
\E_{n}[s_{\lambda}(w)]=O\left(\frac{1}{n^{2k\scl(w)}}\right).\label{eq:mpsclbound}
\end{equation}
\end{thm}

In fact, \cite[Cor. 1.11]{MageePuder1} says that for every $w\in[\F_{r},\F_{r}]$,
there is some $\lambda$ such that the bound (\ref{eq:mpsclbound})
is saturated, but we do not use this here.

Combining (\ref{eq:duncan-howie}) with (\ref{eq:mpsclbound}) gives
\[
\E_{n}[s_{\lambda}(w)]=O\left(n^{-k}\right)
\]
as required.

\section{Appendix: Extension of the method of Random Unitary Representations
of Surface Groups II\label{sec:Appendix:-Extension-of}}

\global\long\def\SU{\mathsf{SU}}%
\global\long\def\p{\mathfrak{p}}%
\global\long\def\NN{\mathcal{N}}%
\global\long\def\hN{\hat{\mathcal{N}}}%
\global\long\def\Wg{\mathrm{Wg}}%
\global\long\def\Res{\mathrm{Res}}%
\global\long\def\ABG{\mathrm{ABG}}%
\global\long\def\zd{\mathrm{zd}}%
\global\long\def\SSTab{\mathcal{SST}}%
\global\long\def\FF{\mathcal{F}}%
\global\long\def\q{\mathfrak{q}}%
\global\long\def\PSU{\mathsf{PSU}}%
\global\long\def\irr{\mathrm{irr}}%
\global\long\def\su{\mathfrak{su}}%
\global\long\def\Ad{\mathrm{Ad}}%
\global\long\def\m{\mathbf{m}}%
\global\long\def\x{\mathbf{x}}%
\global\long\def\LR{\mathsf{LR}}%
\global\long\def\Q{\mathbf{Q}}%
\global\long\def\Z{\mathbf{Z}}%
\global\long\def\T{\mathcal{T}}%
\global\long\def\tkld{\dot{\T}_{n}^{k,\ell}}%
\global\long\def\tkl{\T_{n}^{k,\ell}}%
\global\long\def\match{\mathsf{MATCH}}%
\global\long\def\surfaces{\mathsf{surfaces}}%
\global\long\def\tkldc{\check{\dot{\mathcal{T}}}_{n}^{k,\ell}}%
\global\long\def\tklc{\check{\mathcal{T}}_{n}^{k,\ell}}%
\global\long\def\II{\mathfrak{I}}%
\global\long\def\j{\mathfrak{j}}%

\global\long\def\u{\mathbf{u}}%
\global\long\def\v{\mathbf{v}}%
\global\long\def\r{\mathbf{r}}%
\global\long\def\s{\mathbf{s}}%
\global\long\def\UU{\mathbf{U}}%
\global\long\def\V{\mathbf{V}}%
\global\long\def\RR{\mathbf{\mathbf{R}}}%
\global\long\def\SS{\mathcal{S}}%

In this section we explain the extension of the methodology of \cite{MageeRURSGII}
to bound the integrals $\E_{n}[s_{\lambda,\mu}(w)]$ defined in (\ref{eq:exp-trace-word-def}).
We import some results from \cite{MageeRURSGII}. 

\subsection*{Background }

In this section, fix $\lambda\vdash k$ and $\mu\vdash\ell$, and
assume $n\geq k+\ell$. Let $D_{\lambda,\mu}(n)=s_{\lambda,\mu}(1)$.\textbf{
}Write $\chi_{\lambda}$ for the character of $S_{k}$ associated
to the irreducible representation $W^{\lambda}$ that corresponds
to $\lambda$. Let $d_{\lambda}\eqdf\chi_{\lambda}(\id)=\dim W^{\lambda}$.
Given $\lambda\vdash k$, the element
\[
\p_{\lambda}\eqdf\frac{d_{\lambda}}{k!}\sum_{\sigma\in S_{k}}\chi_{\lambda}(\sigma)\sigma\in\C[S_{k}]
\]
 is the central projection in $\C[S_{k}]$ to the $W^{\lambda}$-isotypic
component. We view $S_{k}\times S_{\ell}$ as a subgroup of $S_{k+\ell}$
in the standard way. 

Let 
\[
\p_{\lambda\otimes\mu}\eqdf\p_{\lambda}\p'_{\mu}
\]
 where $\p'_{\mu}$ is the image of $\p_{\mu}$ under the inclusion
$S_{\ell}\leq S_{k}\times S_{\ell}\leq S_{k+\ell}$. We define Young
subgroups, for $\lambda\vdash k$ 
\[
S_{\lambda}\eqdf S_{\lambda_{1}}\times S_{\lambda_{2}}\times\cdots\times S_{\lambda_{\ell(\lambda)}}\leq S_{k}.
\]
Let 
\[
\tkl\eqdf\left(\C^{n}\right)^{\otimes k}\otimes\left(\left(\C^{n}\right)^{\vee}\right)^{\otimes\ell}.
\]
For $I=(i_{1},\ldots,i_{k+\ell})$ let 
\begin{align*}
I'(I;\pi) & \eqdf i_{\pi(1)},\ldots,i_{\pi(k)},\\
J'(I;\pi) & \eqdf i_{\pi(k+1)},\ldots,i_{\pi(k+\ell)}.
\end{align*}
For $\pi\in S_{k+\ell}$ let 
\[
\Phi(\pi)\eqdf\sum_{I=(i_{1},\ldots,i_{k}),J=(j_{k+1},\ldots,j_{k+\ell})}e_{I'(I\sqcup J;\pi)}^{J}\otimes\check{e}_{I}^{J'(I\sqcup J;\pi)}\in\End(\tkl)
\]
and extend the map $\Phi$ linearly to $\C[S_{k+\ell}]$. 

Recall the definition of the Weingarten function from (\ref{eq:Wg-def}).
Let

\begin{align}
z & \eqdf\sum_{\tau\in S_{k+\ell}}z(\tau)\tau\label{eq:z-theta-def}\\
 & \eqdf\frac{[S_{k}:S_{\lambda}][S_{\ell}:S_{\mu}]}{d_{\lambda}d_{\mu}}\p_{\lambda\otimes\mu}\left(\sum_{\sigma\in S_{\lambda}\times S_{\mu}}\sigma\right)\p_{\lambda\otimes\mu}\Wg_{n,k+\ell}\in\C[S_{k+\ell}].\nonumber 
\end{align}
One has the following bound on the coefficients of $z$ \cite[Lemma 3.3]{MageeRURSGII}
\begin{equation}
z(\tau)=O_{k,\ell}(n^{-k-\ell-\|\tau\|_{k,\ell}}).\label{eq:z-bound}
\end{equation}

Let 
\[
\q\eqdf D_{\lambda,\mu}(n)\Phi(z).
\]

\begin{thm}
\label{thm:projection}~
\begin{enumerate}
\item The operator $\q$ is an orthogonal projection with $\U(n)$-invariant
image that is isomorphic to $V^{\lambda,\mu}$ as a $\U(n)$-representation.
\item For any $i_{1},\ldots,i_{k+\ell},j_{1},\ldots,j_{k+\ell}$ and any
$p\in[k]$,$q\in [k+1,k+\ell]$
\begin{align}	
\sum_{u}\q_{(i_{1}\cdots  i_{k+\ell}),(j_{1}\cdots j_{p-1}uj_{p+1}\cdots j_{q-1} u j_{q+1} \cdots j_{k+\ell})}&=0,\\
\sum_{u}\q_{(j_{1}\cdots j_{p-1}uj_{p+1}\cdots j_{q-1} u j_{q+1} \cdots j_{k+\ell}),(i_{1}\cdots  i_{k+\ell})}&=0 .
\end{align}

\end{enumerate}
\end{thm}

Part 1 is the combination of \cite[Lemma 2.3, eq. (3.11), Prop. 3.2, eq. (3.12)]{MageeRURSGII}.

Part 2 arises from the fact that $\q$ is zero on the orthocomplement
to the contraction free subspace $\tkld$ from \cite{MageeRURSGII}.

\subsection*{Combinatorial integration}

We write $w$ in reduced form:
\begin{equation}
w=f_{1}^{\epsilon_{1}}f_{2}^{\epsilon_{2}}\ldots f_{q}^{\epsilon_{q}},\quad\epsilon_{u}\in\{\pm1\},\,f_{u}\in\{x_{i}\},\label{eq:combinatorial-word}
\end{equation}
where if $f_{u}=f_{u+1}$, then $\epsilon_{u}=\epsilon_{u+1}$. For
$f\in\{x_i\}$ let $p_{f}$ denote the number of occurrences of
$f^{+1}$ in (\ref{eq:combinatorial-word}). 
The expression (\ref{eq:combinatorial-word}) implies that for $u\eqdf(u_{i}:i\in[r])\in\U(n)^{r}$,
\begin{align}
s_{\lambda,\mu}(w(u)) & =\Tr_{\tkl}\left(\q f_{1}^{\epsilon_{1}}\q f_{2}^{\epsilon_{2}}\q\ldots\q f_{q}^{\epsilon_{q}}\right)\label{eq:word-trace}\\
 & =\sum_{I_{j},K_j\in[n]^{k},J_{j},L_j\in[n]^{\ell}}\prod_{u=1}^{q}\q_{K_{i}\sqcup L_{i},I_{i+1}\sqcup J_{i+1}}\nonumber \\
 & (u_{f_{1}}^{\epsilon_{1}})_{I_{1}\sqcup J_{1},K_{1}\sqcup L_{1}}(u_{f_{2}}^{\epsilon_{2}})_{I_{2}\sqcup J_{2},K_{2}\sqcup L_{2}}\cdots(u_{f_{q}}^{\epsilon_{q}})_{I_{q}\sqcup J_{q},K_{q}\sqcup L_{q}}\label{eq:expansion1}\\
 & =\sum_{\pi_{1},\ldots,\pi_{1}\in S_{k+\ell}}\prod_{i=1}^{q}z(\pi_{i})\sum_{I_{j},K_j\in[n]^{k},J_{j},L_j\in[n]^{\ell}}\q_{K_{i}\sqcup L_{i},I_{i+1}\sqcup J_{i+1}}\nonumber \\
 & (u_{f_{1}}^{\epsilon_{1}})_{I_{1}\sqcup J_{1},K_{1}\sqcup L_{1}}(u_{f_{2}}^{\epsilon_{2}})_{I_{2}\sqcup J_{2},K_{2}\sqcup L_{2}}\cdots(u_{f_{q}}^{\epsilon_{q}})_{I_{q}\sqcup J_{q},K_{q}\sqcup L_{q}}\nonumber \\
 & \mathbf{1}\left\{ K_{i}\sqcup J_{i+1}=(I_{i+1}\sqcup L_{i})\circ\pi_{i}\,:\,i\in[q]\,\right\} .\label{eq:expansion2}
\end{align}
In the last line above the indices run mod $q$. Each product of matrices
here can be integrated using the Weingarten calculus.

We view all $I_{u}$ etc as functions from \emph{indices} (the domain)
to $[n]$. We view all domains for distinct $u$ as disjoint as possible. There are however, fixed matchings between the domains of each 
\[
I_{u}\sqcup J_{u}\text{ and }K_{u}\sqcup L_{u}.
\]
These will come into play later.  We now define sub-collections of all the indices that are treated
as `the same type' by the Weingarten calculus. 

For each $i\in[r]$ let:
\begin{itemize}
\item $\SS_{i}$ be all indices of $I_{u}$ such that $f_{u}=x_{i}$ and
$\epsilon_{u}=+1$ and indices of $L_{u}$ such that $f_{u}=x_{i}$
and $\epsilon_{u}=-1$,
\item $\SS_{i}^{*}$ be all indices of $J_{u}$ such that $f_{u}=x_{i}$
and $\epsilon_{u}=+1$ and indices of $K_{u}$ such that $f_{u}=x_{i}$
and $\epsilon_{u}=-1$,
\item $\T_{i}$ be all indices of $K_{u}$ such that $f_{u}=x_{i}$ and
$\epsilon_{u}=+1$ and indices of $J_{u}$ such that $f_{u}=x_{i}$
and $\epsilon_{u}=-1$,
\item $\T_{i}^{*}$ be all indices of $L_{u}$ such that $f_{u}=x_{i}$
and $\epsilon_{u}=+1$ and indices of $I_{u}$ such that $f_{u}=x_{i}$
and $\epsilon_{u}=-1$.
\end{itemize}
Let $\M$ denote the set of data consisting of 
\begin{itemize}
\item for each $i\in[r]$, $\sigma_{i}$ a bijection from $\SS_{i}$ to
$\SS_{i}^{*}$,
\item for each $i\in[r]$, $\tau_{i}$ a bijection from $\T_{i}$ to $\T_{i}^{*}$.
\end{itemize}
Doing the integral of terms in either (\ref{eq:expansion1}) or (\ref{eq:expansion2})
replaces 
\[
(u_{f_{1}}^{\epsilon_{1}})_{I_{1}\sqcup J_{1},K_{1}\sqcup L_{1}}(u_{f_{2}}^{\epsilon_{2}})_{I_{2}\sqcup J_{2},K_{2}\sqcup L_{2}}\cdots(u_{f_{q}}^{\epsilon_{q}})_{I_{q}\sqcup J_{q},K_{q}\sqcup L_{q}}
\]
by 
\[
\sum_{\D\in\M}\prod_{i\in[r]}\Wg_{k+\ell}(\sigma_{i}\tau_{i}^{-1})\mathbf{1}\{\D\text{\, `respected' by \ensuremath{I_{u},J_{u},K_{u},L_{u}}\}.}
\]
Before going on, we make a key argument. Suppose that some index of
$I_{2}$ is matched to an index of $J_{2}$ by $\D$, for example
the two first indices. Then the result of integrating (\ref{eq:expansion1}),
the terms corresponding to $\D$ contain all contain factors 
\[
\sum_{a=I_{2}(1)=J_{2}(1)}\q_{K_{1}\sqcup L_{1},I_{2}\sqcup J_{2}}
\]
where all indices but the first indices of $I_{2}$ and $J_{2}$ are
frozen (conditioned upon). This is zero by Theorem \ref{thm:projection}
part 2. 

This means, going back to (\ref{eq:expansion2}), if we define $\M^{*}$
to be the subset of $\M$ such that 
\begin{itemize}
\item $\sigma_{i}$ never matches indices $I_{u}$ to those of $J_{u}$
for any $u$,
\item $\sigma_{i}$ never matches indices of $L_{u}$ to those of $K_{u}$
for any $u$,
\item $\tau_{i}$ never matches indices of $K_{u}$ to those of $L_{u}$
for any $u$,
\item $\tau_{i}$ never matches indices of $J_{u}$ to those of $I_{u}$
for any $u$,
\end{itemize}
then we obtain
\begin{equation}
\E_{n}[s_{\lambda,\mu}(w(u))]=\sum_{\pi_{1},\ldots,\pi_{q}\in S_{k+\ell}}\sum_{\D\in\M^{*}}\prod_{i=1}^{q}z(\pi_{i})\prod_{i\in[r]}\Wg_{k+\ell}(\sigma_{i}\tau_{i}^{-1})\NN(\pi_{1},\ldots,\pi_{q},\D)\label{eq:fukll-expansion}
\end{equation}
where $\NN(\pi_{1},\ldots,\pi_{q},\D)$ is the number of choices of
$I,J,K,L$ such that 
\begin{align}
K_{u}\sqcup J_{u+1} & =(I_{u+1}\sqcup L_{u})\circ\pi_{u}\,:\,u\in[q],\,\label{eq:pi-matchings}\\
\text{indices matched by \ensuremath{\D}} & \text{ have the same value}.\nonumber 
\end{align}

\subsection*{Surface construction}

Now we construct a surface from $\D\in \M^*$ as follows. Begin with a vertex for every index. Add an edge (called $\D$-edge)
between all matched indices (by $\D$). Add an edge (called $w$-edge)
between indices paired by the fixed identifications
\[
I_{u}\sqcup J_{u}\cong[k+\ell]\cong K_{u}\sqcup L_{u}
\]
and direct this edge from the indices on the left hand side above
to those on the right hand side.

Add also an edge (called $\pi$-edge) between indices matched by (\ref{eq:pi-matchings}).
We now have a trivalent graph.

We now glue two types of discs to this graph following \cite[\S\S 4.1]{MageeRURSGII}.
The boundaries of the discs are glued along two types of cycles in
the graph:
\begin{description}
\item [{Type-I}] Cycles that alternate between $\pi$-edges and $\D$ edges.
Such cycles are disjoint.
\item [{Type-II}] Cycles that alternate between $w$-edges and $\D$-edges.
Again, such cycles are disjoint.
\end{description}
The resulting glued discs therefore meet only along the $\D$-edges
and the resulting total object is a topological surface we call 
\[
\Sigma(\D,\{\pi_{i}\}).
\]
The boundary cycles of this surface alternate between $w$-edges and
$\pi$-edges.

For $\sigma\in S_{k+\ell}$, let $\|\sigma\|_{k,\ell}$ denote the
minimum $m$ for which 
\[
\sigma=\sigma_{0}t_{1}t_{2}\cdots t_{m}
\]
 where $\sigma_{0}\in S_{k}\times S_{\ell}$ and $t_{1},\ldots,t_{m}$
are transpositions in $S_{k+\ell}$. The analog of \cite[Thm. 4.3]{MageeRURSGII}
is that, after an elementary calculation using (\ref{eq:z-bound})
and bounds for the Weingarten function one obtains
\[
\prod_{i=1}^{q}z(\pi_{i})\prod_{i\in[r]}\Wg_{k+\ell}(\sigma_{i}\tau_{i}^{-1})\NN(\pi_{1},\ldots,\pi_{q},\D)\ll_{w,k,\ell}n^{-\sum_{i}\|\pi_{i}\|_{k,\ell}}n^{\chi(\Sigma(\D,\{\pi_{i}\}))},
\]
hence from (\ref{eq:fukll-expansion})
\[
\E_{n}[s_{\lambda,\mu}(w(u))]\ll_{w,k,\ell}\sum_{\pi_{1},\ldots,\pi_{q}\in S_{k+\ell}}\sum_{\D\in\M^{*}}n^{-\sum_{i}\|\pi_{i}\|_{k,\ell}}n^{\chi(\Sigma(\D,\{\pi_{i}\}))}.
\]

\subsection*{Surfaces with large contribution}

It is shown in \cite[\S\S 4.2]{MageeRURSGII} --- the same proof
applies without change to the current setting --- that given any
\[
\pi_{1},\ldots,\pi_{q}\in S_{k+\ell},\,\D\in\M^{*},
\]
it is possible to modify these so that
\begin{equation}
\pi'_{1},\ldots,\pi'_{q}\in S_{k}\times S_{\ell},\,\,\,\,\,\sum_{i}\|\pi'_{i}\|_{k,\ell}=0\label{eq:first-surgery}
\end{equation}

\begin{equation}
\D'=\{\sigma'_{i},\tau_{i}'\}\in\M^{*}\,\text{with \ensuremath{\sigma'_{i}=\tau'_{i}} for all \ensuremath{i\in[r]}}\label{eq:second-surgery}
\end{equation}
and there exists an inequality between corresponding terms
\[
n^{\chi(\Sigma(\D',\{\pi'_{i}\}))}\geq n^{-\sum_{i}\|\pi_{i}\|_{k,\ell}}n^{\chi(\Sigma(\D,\{\pi_{i}\}))}.
\]

The condition (\ref{eq:second-surgery}) means all type-II cycles
are now rectangles with two (non-consecutive) edges in the boundary;
we now replace each rectangle with an arc connecting the two boundary
segments of the rectangle.

\subsection*{Connection to transverse maps}

We now create a class of transverse map on the surface $\Sigma(\D',\{\pi'_{i}\})$
as follows (cf. Definition \ref{rem:ribbon-graph}). The ribbon graph structure is the one dictated by the arcs we just
prior constructed. By construction, they cut the surface into discs.
For each boundary component, place a marked point in some (it does
not matter) $\pi$-edge that arose from $\pi'_{q}$. Recall the $CW$-complex
rose $R_{r}^{0}$ with one point (call it $z_{i}$) in the interior
of the circle corresponding to $x_{i}$. If an arc arose from a rectangle
with edges that arose from $\sigma'_{i}=\tau'_{i}$ then we declare
our function to take the constant value $z_{i}$ on this arc and the
transverse map will traverse the arc from the $\sigma'_{i}$ side
to the $\tau'_{i}$ side. 

The property (\ref{eq:first-surgery}) implies that for every boundary
component of the surface, the $w$-edges are all directed the same
way along this boundary component and hence give an orientation to
the boundary. With respect to this orientation, the isotopy class
of SF transverse map satisfies (\ref{eq:boundary-map-form}).

The fact that $\D'\in\M^{*}$, rather than $\M$, implies that the
class of transverse map has the crucial \textbf{Forbidden} \textbf{Matchings}
property. This completes the proof of Proposition \ref{prop:Magee-decay-rate}.

\bibliographystyle{amsalpha}
\bibliography{strong_convergence}

\noindent Michael Magee, \\
Department of Mathematical Sciences, Durham University, Lower Mountjoy,
DH1 3LE Durham, UK\\
IAS Princeton, School of Mathematics, 1 Einstein Drive, Princeton
08540, USA\\
\texttt{michael.r.magee@durham.ac.uk}~\\

\noindent Mikael de la Salle, \\
Institut Camille Jordan, CNRS, Universit\'{e} Lyon 1, France\\
IAS Princeton, School of Mathematics, 1 Einstein Drive, Princeton
08540, USA\\
\texttt{delasalle@math.univ-lyon1.fr}\\

\end{document}